\documentclass{amsart}
\usepackage{amssymb}

\usepackage{amscd}
\newtheorem{theorem}{Theorem}[section]
\newtheorem{corollary}[theorem]{Corollary}

\newtheorem{lemma}[theorem]{Lemma}
\newtheorem{proposition}[theorem]{Proposition}

\newtheorem{remark}{Remark}[section]

\newtheorem{definition}{Definition}[section]

\input{tcilatex}

\begin{document}
\title[Numerical Criteria]{Criteria for very ampleness of rank two vector bundles over ruled surfaces}
\author{Alberto Alzati}
\address{Dipartimento di Matematica Univ. di Milano\\
via C.\ Saldini 50 20133-Milano (Italy)}
\email{alberto.alzati@unimi.it}
\thanks{This work is within the framework of the national research project ''Geomety
of Algebraic Varieties'' Cofin 2006 of MIUR}
\author{Gian Mario Besana}
\address{College of Computing and Digital Media, De Paul University - 243 S Wabash
Chicago IL 60604 USA}
\email{gbesana@cs.depaul.edu}
\date{february, 12 2009}
\maketitle

\begin{abstract}
Very ampleness criteria for rank $2$ vector bundles over smooth, ruled
surfaces over rational and elliptic curves are given. The criteria are then
used to settle open existence questions for some special threefolds of low
degree.
\end{abstract}

\section{Introduction.}

A vector bundle $\mathcal{E}$ over a smooth algebraic variety $Y$ is said to
be very ample if the tautological line bundle $\mathcal{O}_{\Bbb{P}(\mathcal{%
E})}(1)$ is very ample on the projectivized bundle $\Bbb{P}(\mathcal{E}).$
Very ampleness of $\mathcal{E}$ is therefore equivalent to the existence of
a projective smooth manifold $X=\Bbb{P}(\mathcal{E})$ embedded as a linear
scroll on $Y.$

Although it is in general impossible to give a numerical characterization of
very ampleness, one can try to find sufficient numerical conditions to
guarantee it when the Picard group of $Y$ is particularly simple.

In this paper some classical ideas are revisited in order to give some very
ampleness criteria for rank $2$ vector bundles $\mathcal{E}$ over smooth,
ruled surfaces on rational and elliptic curves.

In Section 3, classical ideas on obtaining very ampleness criteria by
lifting of sections from appropriately chosen divisors are revisited in our
context. They are applied to obtain the very ampleness of a family of rank $%
2 $ vector bundles over $\Bbb{P}^{2}$, see Corollary \ref{bundledeg11g6onP2}
, and, in particular, an existence result for $3$-dimensional scrolls over $%
\Bbb{P}^{2}$, of degree $11$ and genus $6,$ left as an open question in \cite
{bb1}, see Remark \ref{edeg11g6onP2}. Section 4 presents a purely numerical
very ampleness criterion for rank $2$ vector bundles over rational ruled
surfaces, see Theorem \ref{Valma}, with an example of its application.
Section 5 deals with the case of ruled surfaces over elliptic curves.
Section 6 contains a very ampleness criterion criterion for a very special
class of vector bundles $\mathcal{E}$ on $\mathbf{F}_{1}$, the Hizebruch
surface with invariant $e=1.$ The criterion is partially based upon a new
observation that relates the very ampleness of $\mathcal{E}$ with whether
points in the zero locus of a generic section of $\mathcal{E}$ are in
general position on $Y.$ Results from section 6 are applied to establish
further existence and non existence results for threefolds scrolls over $%
\mathbf{F}_{1}$, of low degree, previously left as open problems in \cite
{fl2} and \cite{bb1}.

\textbf{Acknowledgements:} The authors wish to thank E. Arrondo for some
helpful conversations around the zero locus of generic sections of rank $2$
vector bundles over surfaces and the referee for the very careful revision
of the paper.

\section{Notation and background material}

The ground field is fixed to be $\Bbb{C},$ and $\Bbb{P}^{n}$ denotes the $n$
-dimensional complex projective space. The focus of this work is on rank-two
vector bundles over smooth surfaces. In this context we fix the following
notation:

\begin{itemize}
\item  $\mathcal{E}:$ a rank-$r$ vector bundle over a smooth variety $Y,$ $%
\dim {Y}\ge 2;$

\item  $c_{i}(\mathcal{E}):$ the $i$-th Chern class of $\mathcal{E};$

\item  $X=\Bbb{P}(\mathcal{E}):$ the projectivization of $\mathcal{E};$

\item  $\pi :\Bbb{P}(\mathcal{E})\to Y$ the natural projection onto the base;

\item  $T:$ the tautological line bundle of $X,$ i.e. ${}\mathcal{O}_{X}({T}%
)={}\mathcal{O}_{\Bbb{P}(\mathcal{E})}(1);$

\item  $|T|:$ linear system of effective divisors linearly equivalent to $T;$

\item  $\Gamma :$ projectivization of the restriction $\mathcal{E}_{|\gamma
},$ where $\gamma $ is a smooth curve on $Y;$

\item  $\mathbf{F}_{e}:$ rational ruled surface of invariant $e\geq 0,$ i.e.
$\Bbb{P}(\mathcal{O}_{\Bbb{P}^{1}}\oplus \mathcal{O}_{\Bbb{P}^{1}}(-e));$

\item  $\rho :$ $\mathbf{F}_{e}\rightarrow \Bbb{P}^{1}:$ the natural
projection onto the base;

\item  $C_{0},f:$ standard generators of $\text{\textrm{Num}}(\mathbf{F}%
_{e})\simeq \text{\textrm{Pic}}(\mathbf{F}_{e});$

\item  $F:$ projectivization of the restriction $\mathcal{E}_{|f}$ where $f$
is a fibre of $Y$ when $Y=\mathbf{F}_{e};$

\item  $(\sigma )_{0}:$ zero locus of the section $\sigma $ of a vector
bundle;

\item  $\equiv $ $:$ numerical equivalence of divisors;

\item  $K_{S}:$ canonical divisor of the smooth surface $S;$

\item  $X^{[t]}:$ the Hilbert scheme of zero-subschemes of $X$ of length $t.$

\item  $\Bbb{T}_{P}(M):$ the holomorphic tangent space to an analytic
manifold $M$ at a point $P.$
\end{itemize}

Cartier divisors on smooth projective varieties, their associated line
bundles and invertible sheaves of their holomorphic sections are used with
no distinction. Mostly additive notation is used for their group. Given a
divisor (line bundle) $D$ on a smooth projective variety, $|D|$ denotes the
complete linear system of effective divisors linearly equivalent
(associated) to $D.$ Given any subvariety $S$ in $X$ and a line bundle $L\in
\text{\textrm{Pic}}(X),$ we denote by ${L}_{\mid _{S}}$ the restriction of $%
L $ to $S,$ i.e. ${L}_{\mid _{S}}=L\otimes \mathcal{O}{}_{S}.$

Let $X$ and $T$ be as above. For any smooth surface $\Sigma $ contained in $%
X $ let $|T_{|\Sigma }|$ be the complete linear system associated to ${T}
_{|\Sigma },$ i.e. given by $H^{0}(T\otimes \mathcal{O}{}_{\Sigma }).$ Let $%
|T|_{|\Sigma }$ be the restriction of the linear system $|T|$ to $\Sigma ,$
i.e. given by the image of the restriction map $r:H^{0}(X,T)\to H^{0}(\Sigma
,{T}_{\mid _{\Sigma }}).$ Then $|T_{|\Sigma }|\supseteq |T|_{|\Sigma }$ and
equality holds if $h^{1}(X,T-\Sigma )=0.$

If $\Sigma $ is reducible as the union of two smooth surfaces $S_{1}\cup
S_{2},$ intersecting transversely only along a smooth (possibly reducible)
curve $C,$ it is: $\{(\sigma _{1},\sigma _{2})|$ $\sigma _{i}\in
|T_{|S_{i}}| $ $i=1,2,$ $\sigma _{1_{|C}}=\sigma _{2_{|C}}\}=|T_{|\Sigma
}|\subseteq |T_{|S_{1}}|\oplus |T_{|S_{2}}|=\{(\sigma _{1},\sigma
_{2})|\sigma _{i}\in |T_{|S_{i}}|$ $i=1,2\},$ while $|T|_{|\Sigma
}=\{(\sigma _{1},\sigma _{2})|\sigma _{i}\in |T_{|S_{i}}|$ $i=1,2$ and there
exists $\tau \in |T|$ such that $\tau _{|S_{i}}=\sigma _{i}$ $i=1,2\}.$

Let $\xi \in X^{[t]}.$ A subvariety $S$ is said to \textit{pass through} $%
\xi $ if and only if $S$ contains $\xi $ scheme theoretically. If $t=\mathrm{%
length}(\xi )=2$ and $\mathrm{Supp}(\xi )$ consists of two distinct points,
scheme theoretic inclusion is equivalent to ordinary inclusion. If $t=%
\mathrm{length}(\xi )=2,$ $\mathrm{Supp}(\xi )$ consists of one point $P,$
and $X$ and $S$ are smooth at $P$, let $\underline{q}\in \Bbb{T}_{P}(X)$
denote the tangent direction at $P$ specified by $\xi .$ Then scheme
theoretic inclusion is equivalent to $P\in S$ and $\underline{q}\in \Bbb{T}
_{P}(S).$

Let $L$ be a line bundle on a variety $X$ and let $\xi \in X^{[t]}.$ Let $%
V\subset H^{0}(L)$ be a subspace of sections and let $|V|$ be the associated
linear system. The expression $|V|$ \textit{separates} $\xi $ is used to
mean that the restriction map $V\to H^{0}({}\xi )$ is surjective. In this
language a line bundle $L$ is very ample if and only if the associated
complete linear system separates every $\xi \in X^{[2]}$ .

If $\xi \in X^{[t]}$ is a reduced, $0$-dimensional, scheme, we often
identify the scheme itself with its support. For example we refer to
''points of $\xi "$ to mean points of $\mathrm{Supp}(\xi )$.

\begin{lemma}
\label{nolines} Let $Y$ be a ruled surface of invariant $e,$ over a smooth
curve $C$ of genus $g.$ Let $x$ be an integer and assume the line bundle $%
C_{0}+xf$ is very ample. If $\ell \subset Y$ is a line on $Y$ in the
embedding given by $|C_{0}+xf|,$ then either $\ell =f$ or $g=0,$ $x=e+1,$
and $\ell =C_{0}.$
\end{lemma}

\begin{proof}
As $C_{0}+xf$ is very ample, it is in particular ample and therefore, (see
\cite[Corollary V.2.18, Proposition V.2.20, Proposition V.2.21]{h}) it must
be
\begin{equation}
x>\begin{cases} e \text{ if } e \ge 0\\ \frac{e}{2}\text{ if } e < 0.\\
\end{cases}  \label{ampleness}
\end{equation}
For an irreducible divisor $\ell \in |aC_{0}+bf|$ to be a line in the
embedding given by $C_{0}+xf$ it must be $(C_{0}+xf)(aC_{0}+bf)=1$ i.e.
\begin{equation}
1=ax-ae+b.  \label{bealine}
\end{equation}
By considering that the arithmetic genus of $\ell $ must be zero, one can
easily check that the necessary conditions for $\ell $ to be an irreducible
divisor, contained in
\cite[Corollary 2.18, Proposition 2.20, Proposition 2.21]{h}, are
incompatible with (\ref{ampleness}), (\ref{bealine}), and the very ampleness
of $C_{0}+xf$ unless $\ell $ is as in the statement.
\end{proof}

\section{Classical ideas}

The following notation will be fixed throughout the paper. Let $Y$ be a
smooth algebraic surface. Let $\mathcal{E}$ be a rank $2$ vector bundle over
$Y$, and let $X=\Bbb{P}(\mathcal{E})$. Let $\pi :X\to Y$ be the natural
projection and let $T$ be the tautological line bundle$.$ Let $A\in \mathrm{%
\ Pic}(Y)$ and let $D=T+\pi ^{*}A\in \mathrm{Pic}(X).$ The first Proposition
in this section is a simple adaptation to our context of a classical lifting
of sections approach to prove very ampleness.

\begin{proposition}
\label{uno}With the notation fixed in this section, let $D_{\epsilon
}=\epsilon T+\pi ^{*}(A),$ where $\epsilon =0,1$ and assume:

\begin{itemize}
\item[a)]  for all $\xi \in X^{[2]}$ there exists a smooth, irreducible
element $S\in |D_{\epsilon }|,$ passing through $\xi ;$

\item[b)]  $h^{1}(X,(1-\epsilon )T-\pi ^{*}(A))=0;$

\item[c)]  $T_{|S}$ is very ample on $S.$
\end{itemize}

Then $T$ is very ample on $X$.
\end{proposition}

\begin{proof}
Given any $\xi \in X^{[2]},$ assumption a) gives an element $S\in
|D_{\epsilon }|$ passing through it. By assumption c) we can separate $\xi $
on $S$ by elements of $|T_{|S}|.$ Consider the sequence: $0\rightarrow
H^{0}(X,T-D_{\epsilon })\rightarrow H^{0}(X,T)\rightarrow
H^{0}(S,T_{|S})\rightarrow H^{1}(X,T-D_{\epsilon })=H^{1}(X,(1-\epsilon
)T-\pi ^{*}A)=0$. Assumption b) allows us to lift the separating sections on
$S$ to sections of $T$ separating $\xi $ on $X.$ Notice that when $\epsilon
=1$ assumption b) can be restated as $h^{1}(Y,-A)=0.$
\end{proof}

As an immediate application of Proposition \ref{uno}, when $\epsilon =0,$
one obtains the very ampleness of a family of rank-$2$ vector bundles over $%
\Bbb{P}^{2}.$

\begin{corollary}
\label{bundledeg11g6onP2} There exists very ample vector bundles $\mathcal{E}
$ of rank $2$ over $Y=\Bbb{P}^{2}$ given by non trivial extensions
\begin{equation*}
0\rightarrow \mathcal{O}_{\Bbb{P}^{2}}(1)\rightarrow \mathcal{E}\rightarrow
\mathcal{O}_{\Bbb{P}^{2}}(4)\otimes \mathcal{I}_{\eta }\rightarrow 0,
\end{equation*}
where $\eta \in \Bbb{P}{^{2}}^{[10]}$ consists of $10$ distinct points in
general position.
\end{corollary}

\begin{proof}
Let $\eta \in \Bbb{P}{^{2}}^{[10]}$ consist of $10$ distinct points in
general position. Notice that $Ext^{1}(\mathcal{O}_{\Bbb{P}^{2}}(3)\otimes
\mathcal{I}_{\eta },\mathcal{O}_{\Bbb{P}^{2}})\neq 0$ and there exists a
locally free extension
\begin{equation}
0\rightarrow \mathcal{O}_{\Bbb{P}^{2}}\rightarrow \mathcal{E}^{\prime
}\rightarrow \mathcal{O}_{\Bbb{P}^{2}}(3)\otimes \mathcal{I}_{\eta
}\rightarrow 0,  \label{eprime}
\end{equation}
because $K_{Y}\otimes \mathcal{O}_{\Bbb{P}^{2}}(3)=\mathcal{O}_{\Bbb{P}%
^{2}}, $ so that, for any point $w\in \eta ,$ the natural map $0=H^{0}(Y,%
\mathcal{I}_{\eta })=H^{0}(Y,K_{Y}\otimes \mathcal{O}_{\Bbb{P}%
^{2}}(3)\otimes \mathcal{I}_{\eta })\rightarrow H^{0}(Y,K_{Y}\otimes
\mathcal{O}_{\Bbb{P}^{2}}(3)\otimes \mathcal{I}_{\eta \backslash w})=H^{0}(Y,%
\mathcal{I}_{\eta \backslash w})=0$ is an isomorphism as required, see \cite
{dl}, Theorem 3.13 and its proof. Twisting (\ref{eprime}) by $\mathcal{O}_{%
\Bbb{P}^{2}}({1})$ and setting $\mathcal{E}=\mathcal{E}^{\prime }(1)$ gives
\begin{equation}
0\rightarrow \mathcal{O}_{\Bbb{P}^{2}}(1)\rightarrow \mathcal{E}\rightarrow
\mathcal{O}_{\Bbb{P}^{2}}(4)\otimes \mathcal{I}_{\eta }\rightarrow 0.
\label{edeg11g6onP2}
\end{equation}
Notice that $c_{1}(\mathcal{E})=\mathcal{O}{}_{\Bbb{P}^{2}}({5}),$ and $%
c_{2}(\mathcal{E})=14.$ Let $D=\pi ^{*}({}\mathcal{O}_{\Bbb{P}^{2}}({1})),$
i.e. $\epsilon =0$ and $A=\mathcal{O}{}_{\Bbb{P}^{2}}({1})$ in the notation
of Proposition \ref{uno}. Let ${\mathcal{E}}_{\mid _{\ell }}$ be the
restriction of $\mathcal{E}$ to any line $\ell .$ From (\ref{edeg11g6onP2})
one gets:
\begin{equation}
0\rightarrow \mathcal{O}_{\ell }(1+\varepsilon )\rightarrow {\mathcal{E}}%
_{\mid _{\ell }}\rightarrow \mathcal{O}_{\ell }(4-\varepsilon )\rightarrow 0,
\label{erestrict}
\end{equation}
where $\varepsilon =0,1,2,$ respectively, if $\ell $ passes through $0,1,2$
points of $\eta .$ Sequence (\ref{erestrict}) shows that ${\mathcal{E}}%
_{\mid _{\ell }}$ is very ample, so that assumptions a) and c) of
Proposition \ref{uno} are satisfied. To verify assumption b) of Proposition
\ref{uno} notice that $h^{1}(X,T-D)=h^{1}(\Bbb{P}^{2},\mathcal{E}(-1))=h^{1}(%
\Bbb{P}^{2},\mathcal{E}^{\prime }),$ which, looking at sequence (\ref{eprime}%
), vanishes if and only if $h^{1}(\Bbb{P}^{2},\mathcal{O}{}_{\Bbb{P}^{2}}({3}%
)\otimes \mathcal{I}{}_{\eta })$ vanishes. As $\eta $ consists of $10$
points in general position it is $H^{0}(\mathcal{O}{}_{\Bbb{P}^{2}}({3}%
)\otimes \mathcal{I}{}_{\eta })=0$ and thus the sequence
\begin{equation*}
0\rightarrow \mathcal{O}_{\Bbb{P}^{2}}(3)\otimes \mathcal{I}_{\eta
}\rightarrow \mathcal{O}_{\Bbb{P}^{2}}(3)\rightarrow \mathcal{O}_{\Bbb{P}%
^{2}}(3)\otimes \mathcal{O}_{\eta }\rightarrow 0
\end{equation*}
gives the required vanishing.
\end{proof}

\begin{remark}
\label{deg11g6onP2} Corollary \ref{bundledeg11g6onP2} implies the existence
of a family of 3-dimensional scrolls $(X,L)=(\Bbb{P}(\mathcal{E}),\mathcal{O}%
{}_{\Bbb{P}(\mathcal{E})}(1))$\ embedded by $|\mathcal{O}{}_{\Bbb{P}(%
\mathcal{E})}(1)|$\ in $\Bbb{P}^{7}$\ with $\deg (X)=[c_{1}(\mathcal{E}%
)]^{2}-c_{2}(\mathcal{E})=11,$\ and sectional genus $g(X)=6$ $($see\ also
Section 7). The existence of such threefolds was left as an open question in
\cite{bb1}, Proposition 4.2.3. The Hilbert scheme of a threefold $X\subset
\Bbb{P}^{7}$\ as in Corollary \ref{bundledeg11g6onP2} has an irreducible
component of dimension $83$, of which $X$\ is a smooth point. See \cite{bf},
Proposition 3.1, Corollary 3.3 and Remark 3.4, for details.
\end{remark}

Without any assumption on the positivity of $D$ it would be very difficult
to establish whether assumption a) in Proposition \ref{uno} is satisfied or
not. If $D$ were very ample it is well known, \cite{bs}, that if there is a
zero scheme $\xi $ for which a) fails, then the ambient variety is a surface
containing a $D$-line $\ell $ through $\xi ,$ such that all divisors in $|D|$
passing through $\xi $ are reducible as $\ell +D^{\prime }.$ In our
particular situation, keeping in mind that we are striving for sufficient
numerical criteria, we can relax conditions on $D$ as in the following
Proposition. Before proving it we need a Lemma.

\begin{lemma}
\label{Belt-Som} With the notation fixed in this Section, let $\mathrm{F}$
be the class of a generic fibre of $\pi $ in the Chow ring of $X.$ Let $%
\Sigma $ be a singular element in $|T|.$ Then $\Sigma $ is reducible, more
precisely $\Sigma =\pi ^{-1}(Z)\cup \Sigma ^{\prime }$ where $Z$ is an
effective divisor on $Y$ and $\Sigma ^{\prime }\cdot \mathrm{F}=1,$ and if $%
Z $ is maximal with respect to the previous decomposition then $\Sigma
^{\prime }$ is smooth and irreducible.
\end{lemma}

\begin{proof}
Let $R$ be a singular point of $\Sigma $ and let $\mathrm{F}_{R}$ be the
fibre of $\pi $ passing through $R.$ Then $\mathrm{F}_{R}\simeq \Bbb{P}^{1}$
should intersect $\Sigma $ in $R$ with multiplicity at least $2.$ As $\Sigma
\cdot \mathrm{F}_{R}=1$ in the Chow ring of $X,$ this is possible only if $%
\mathrm{F}_{R}$ is contained in $\Sigma $. If $\mathrm{F}_{R}$ is an
isolated fibre contained in $\Sigma $ then $\Sigma $ would be, locally, the
blow up of $Y$ at $\pi (R),$ hence it would be smooth at $R$. It follows
that there are infinitely many fibres of $\pi $ contained in $\Sigma .$ So
that we can write $\Sigma =\pi ^{-1}(Z)\cup \Sigma ^{\prime }$ where $Z$ is
a suitable effective divisor on $Y$. Obviously $1=\Sigma \cdot \mathrm{F}%
=\Sigma ^{\prime }\cdot \mathrm{F.}$ If $\Sigma ^{\prime }$ is smooth and
irreducible we are done, otherwise we can argue as before for $\Sigma
^{\prime }$ until we get a decomposition with a smooth $\Sigma ^{\prime }\in
|T-\pi ^{*}Z|.$ Such a $\Sigma ^{\prime }$ is also irreducible thank to the
maximality of $Z.$
\end{proof}

\begin{proposition}
\label{due} With the notation fixed in this Section, assume:

\begin{itemize}
\item[a)]  $h^{0}(X,D)>3;$

\item[b)]  for all $B\in \mathrm{Pic}(Y)$ such that $T+\pi ^{*}B$ is
effective and $A-B$ is effective, then $\max \{h^{0}(X,T+\pi ^{*}B),$ $%
h^{0}(Y,A-B)\}<$ $h^{0}(X,D)-2.$
\end{itemize}

Then assumption \textrm{{a)} of Proposition \ref{uno} is satisfied with }$%
\varepsilon =1$\textrm{. }
\end{proposition}

\begin{proof}
As $h^{0}(X,D)>3,$ for all $\xi \in X^{[2]}$ the linear system $%
|V|=|D\otimes \mathcal{I}{}_{\xi }|$ is not empty. Let $S$ be an element of $%
|V|.$ If $S$ is smooth we have nothing to prove, otherwise, by applying
Lemma \ref{Belt-Som} to $\mathcal{E}\otimes \mathcal{O}_{Y}(A),$ we can
write $S=\pi ^{-1}(A^{\prime })\cup S^{\prime }$ where $S^{\prime }\in
|T+\pi ^{*}B|$ is effective, irreducible and smooth for some $B\in \mathrm{%
Pic}(Y)$ and where $A^{\prime }\in |A-B|$ is effective on $Y.$

Notice that $\dim (|V|)\geq \dim (|D|)-2.$ It is enough to show that not
every element of $|V|$ is singular. By contradiction, let us assume that all
elements $S$ of $|V|$ are singular. By Bertini's Theorem $|V|$ has a base
locus \textrm{E} and for any generic $S\in |V|$ we have that $%
Sing(S)\subseteq $\textrm{E}. On the other hand we know that $S=\pi
^{-1}(A^{\prime })\cup S^{\prime }$ as above and $S$ is singular along $\pi
^{-1}(A^{\prime })\cap S^{\prime }$ which can not be the union of a finite
number of fibres and can not contain isolated fibres (see the proof of Lemma
\ref{Belt-Som}), so that there is a maximal curve $C_{S}\subseteq
Sing(S)\subseteq $\textrm{E} such that $\pi [Sing(S)]=\pi (C_{S})$, $\dim
\pi (C_{S})=1,$ $\pi (C_{S})\subseteq A^{\prime }$ and $\pi ^{-1}[\pi
(C_{S})]=\pi ^{-1}(A^{\prime });$ hence $\dim (\mathrm{E})\geq 1$ and $\dim
[\pi (\mathrm{E})]\geq 1.$ Let us consider $\pi (\mathrm{E}).$

If $\dim [\pi (\mathrm{E})]=2$ then $\dim (\mathrm{E})=2$ and $|V|$ $=(T+\pi %
^{*}B)+|\pi ^{*}A-\pi ^{*}B|,$ where $B$ is a suitable divisor of $Y$ such
that $T+\pi ^{*}B$ is effective on $X,$ where $A-B$ is effective on $Y$ and
where $(T+\pi ^{*}B)$ is in the fixed part of $|V|.$ But this is not
possible, by assumption b), as $\dim (|\pi ^{*}A-\pi ^{*}B|)=\dim (|V|)\geq %
\dim (|D|)-2.$

If $\dim [\pi (\mathrm{E})]=1$ let us consider the above curves $C_{S}$ $%
\subseteq $\textrm{E }for any generic $S\in |V|.$ If $C_{S}=C_{\overline{S}}$
for some fixed $\overline{S}=\pi ^{-1}(\overline{A}^{\prime })\cup \overline{%
S}^{\prime }\in |V|$ then $\pi ^{-1}(\overline{A}^{\prime })=\pi ^{-1}[\pi
(C_{\overline{S}})]\subseteq $\textrm{E. }If not, the curves $C_{S}$ $%
\subseteq $\textrm{E} fill some surface $\pi ^{-1}(Z)$ for some effective
divisor $Z$ of $Y$, otherwise $\dim [\pi (\mathrm{E})]=2.$ Hence, in any
case, \textrm{\ E} contains some surface $\pi ^{-1}(Z).$ By choosing $Z$
maximal we have that $|V|=|T+\pi ^{*}B|+(\pi ^{*}A-\pi ^{*}B)$ where $B$ is
a suitable divisor of $Y$ such that $T+\pi ^{*}B$ is effective on $X,$ where
$Z=A-B$ is effective on $Y$ and where $(\pi ^{*}A-\pi ^{*}B)$ is the fixed
part of $|V|.$ But this is not possible, by assumption b), as $\dim (|T+\pi
^{*}B|)=\dim (|V|)\geq \dim (|D|)-2.$
\end{proof}

Proposition \ref{uno}, in case $\epsilon =1,$ still requires the very
ampleness of a line bundle, ${T}_{\mid _{S}},$ on a surface section $S$ of $%
X.$ Even in the most simple situation, when $S$ is isomorphic to the blowing
up of $Y$ at $c_{2}(\mathcal{E}\otimes \mathcal{O}_{Y}(A))$ distinct points,
the very ampleness of $T_{|S}$ depends on the position of these points on $Y$
, and in general we know nothing about it. The following Proposition
circumvents this difficulty by directly showing separation of length-2
zero-schemes on $X.$ The basic technique is the construction of separating
elements as reducible divisors. They are constructed as sums of a\textit{\
horizontal} component and a \textit{vertical} one.

\begin{proposition}
\label{tre}With the notation fixed in this Section, let $B\in \text{\textrm{%
Pic}}(Y)$ be effective. Assume:

\begin{itemize}
\item[$\alpha$)]  for all pairs of distinct points $\{P,Q\}\subset X$

\begin{itemize}
\item[i)]  there exist $\Sigma \in |T+\pi ^{*}A+\pi ^{*}B|,$ which is
reducible as the union of a smooth surface $S_{2}\in |T+\pi ^{*}A|$ passing
through $P$ and not through $Q,$ and a smooth surface $S_{1}\in |\pi ^{*}B|$
passing through $Q;$

\item[ii)]  there exist $\sigma _{1}\in H^{0}(T_{|S_{1}}),$ such that $%
\sigma _{1}(Q)\neq 0$ and $\sigma _{1}(P)={0}$ in case $P\in S_{1}\cap S_{2};
$

\item[iii)]  there exist $\sigma _{2}\in H^{0}(T_{|S_{2}}),$ such that $%
\sigma _{2}(P)=0;$

\item[iv)]  there exist $\sigma \in H^{0}(T_{|\Sigma }),$ such that ${\sigma
}_{\mid _{S_{1}}}=\sigma _{1}$ and ${\sigma }_{\mid _{S_{2}}}=$ $\sigma _{2};
$
\end{itemize}

\item[$\beta $)]  for any point $P\in X$ and for any direction $\underline{q}%
\in \Bbb{T}_{P}(X)$

\begin{itemize}
\item[i)]  there exists $\Sigma \in |T+\pi ^{*}A+\pi ^{*}B|,$ which is
reducible as the union of a smooth surface $S_{2}\in |T+\pi ^{*}A|$ and a
smooth surface $S_{1}\in |\pi ^{*}B|,$ both passing through $P;$

\item[ii)]  there exists $\sigma _{1}\in H^{0}(T_{|S_{1}})$ such that $%
\sigma _{1}(P)=0$ and $(\sigma _{1})_{0}$ is smooth at $P$ with tangent
vector \underline{$t$}$\in \Bbb{T}_{P}(S_{1});$

\item[iii)]  there exists $\sigma _{2}\in H^{0}(T_{|S_{2}})$ such that $\in
\sigma _{2}(P)=0$ and $(\sigma _{2})_{0}$ is smooth at $P$ with tangent
vector \underline{$v$}$\in \Bbb{T}_{P}(S_{2})$ in such a way that \underline{%
$q$}$\notin <\underline{t},\underline{v}>;$

\item[iv)]  there exists $\sigma \in H^{0}(T_{|\Sigma }),$ such that $\sigma
_{|S_{1}}=\sigma _{1}$ and $\sigma _{|S_{2}}=\sigma _{2}.$
\end{itemize}

\item[$\gamma $)]  $h^{1}(Y,-A-B)=0.$
\end{itemize}

Then $T$ is very ample on $X$.
\end{proposition}

\begin{proof}
For any reduced length-2 zero scheme $\{P,Q\}\in X^{[2]}$, assumption $%
\alpha )$ gives a section $\sigma \in H^{0}(T_{|\Sigma })$ such that $\sigma
(P)=0,\sigma (Q)\neq 0,$ and, reversing the roles of $P$ and $Q,$ a section $%
\sigma ^{\prime }\in H^{0}(T_{|\Sigma ^{\prime }})$ such that $\sigma
^{\prime }(Q)=0,$ $\sigma ^{\prime }(P)\neq 0.$ As in the proof of
Proposition \ref{uno}, assumption $\gamma $) allows such $\sigma $ and $%
\sigma ^{\prime }$ to be lifted to elements of $H^{0}(T)$ separating $%
\{P,Q\}.$

For any point $P\in X$ and for any direction \underline{$q$}$\in \Bbb{T}
_{P}(X),$ assumption $\beta )$ gives $\sigma \in H^{0}(T_{|\Sigma })$ such
that $\sigma (P)=0,$ and such that $\Bbb{T}_{P}((\sigma )_{0})$ does not
contain \underline{$q$}. Once again, assumption $\gamma $) allows $\sigma $
to be lifted to $\tau _{1}\in H^{0}(T)$ with ${\tau _{1}}_{|\Sigma }=\sigma $
such that $\tau _{1}(P)=0.$ Notice that $(\tau _{1})_{0}$ is smooth at $P;$
in fact $(\tau _{1})_{0}$ can be singular at $P$ only if it contains the
fibre of $\pi $ through $P,$ but in this case it could not cut $(\sigma
_{1})_{0}$ on $S_{1},$ as in our assumptions, because $(\sigma _{1})_{0}$ is
smooth at $P.$ Hence $\Bbb{T}_{P}((\tau _{1})_{0})=$ $<\underline{t},%
\underline{v}>$ and it does not contain \underline{$q$}. Assumption $\alpha
) $ guarantees the existence of a section $\tau _{2}\in H^{0}(T)$ such that $%
\tau _{2}(P)\neq 0.$ Thus $\{\tau _{1},\tau _{2}\}$ separate the zero scheme
$\{P,\underline{q}\}.$
\end{proof}

We conclude this Section with a simple very ampleness criterion for line
bundles on surfaces. The idea of this proof will be used in \S 4 to give a
criterion for our vector bundles.

\begin{proposition}
\label{cinque} Let $S$ be a smooth projective surface. Let $D,A\in \text{%
\textrm{Pic}}(S)$ and let $z\geq 1$ be a positive integer such that $zA$ is
very ample. Assume:

\begin{itemize}
\item[1)]  for any $\xi \in S^{[2]}$ there exists a smooth, irreducible
curve $\gamma \in |zA|$ passing through $\xi ;$

\item[2)]  $DA\geq (z-1)A^{2}+2p_{a}(A)+1$ ;

\item[3)]  $h^{1}(S,D-zA)=0.$
\end{itemize}

Then $D$ is very ample.
\end{proposition}

\begin{proof}
Let $\xi \in S^{[2]}$ and let $\gamma $ be as in assumption 1). Assumption
2) implies that ${D}_{|\gamma }$ is very ample. In fact $2g(\gamma
)-2=(zA+K_{S})zA,$ hence $2g(\gamma )+1=$ $(zA+K_{S})zA+3$ and $\deg ({D}
_{|\gamma })=zDA\geq 2g(\gamma )+1=$ $(zA+K_{S})zA+3$ because $DA\geq
(z-1)A^{2}+2p_{a}(A)+1=(zA+K_{S})A+3.$ As ${D}_{|\gamma }$ is very ample we
have two sections $\{\sigma _{1},\sigma _{1}^{\prime }\},$ in $H^{0}(\gamma ,%
{D}_{|\gamma }),$ that separate $\xi $ on $\gamma .$ Assumption 3)
guarantees that we can separate $\xi $ on $S$ by lifting $\sigma _{1}$ and $%
\sigma _{1}^{\prime }$ to sections in $H^{0}(S,D).$
\end{proof}

\begin{remark}
The existence of certain surfaces $S\subset \Bbb{P}^{4}$\ of degree $14$\
which are ruled over a genus $2$\ curve, with invariant $e=-2,$\ is a long
standing open problem, see for example \cite{hr}. The candidate very ample
line bundle giving the embedding of $S$\ is predicted to be of numerical
class $7C_{0}-6f.$\ As an application of Proposition \ref{cinque} one can
prove that $D\equiv 7C_{0}+2f$\ is very ample on $S,$\ by setting $A\equiv
C_{0}+5f$, and $z=1.$\ One could then try to solve the original problem by
embedding $S$\ with $D$\ and finding a suitable projection.
\end{remark}

\section{A numerical very ampleness criterion}

As indicated in the introduction, when the base surface $Y$ has a simple
Picard group, one can hope to obtain sufficient conditions to characterize
very ampleness of vector bundles over $Y.$

The heart of this section is Theorem \ref{Valma}. It is obtained using the
same approach described in Section 3 with an interesting twist. The divisor $%
D,$ chosen through the length $2$ zero-scheme that needs separation, is
obtained as a sum of a number of components, all chosen as pull backs of
suitable sections of the base surface. Similarly, the global sections of the
tautological line bundle $T,$ performing the required separation are lifted
from sections glued together from well behaved sections on each of the
components of $D.$

Firstly we have the following

\begin{lemma}
\label{punti-tg} Let $Y=\mathbf{F}_{e}$ and let $A\equiv C_{0}+bf$ be a very
ample divisor on $Y.$ Let $P_{1},...,P_{r}$ be $r$ distinct points on $Y$
such that no two of them lie on the same fibre. Let $F:=\{f_{i}\}$ where $%
f_{i}=\rho ^{*}(\rho (P_{i}).$ If $h^{0}(Y,A)\geq 2r$ then there exists at
least an element $\gamma \in |A|$ passing smoothly through $P_{1},...,P_{r}$
and intersecting transversely any fixed finite set of fibres $\Phi
:=\{\varphi _{j}\}$ with $\Phi \cap F=\emptyset .$
\end{lemma}

\begin{proof}
First of all, let us consider the case in which the set $\Phi $ is empty.
For any point $P_{i}$ let us choose a vector \underline{$q$}$_{i}\in \Bbb{T}
_{P_{i}}(Y)$ pointing towards a direction different from that of $f_{i.}$ By
assumption, there exists at least an element $\gamma \in |A|$ passing
through $P_{1},...,P_{r}$ and having tangent direction \underline{$q$}$_{i}$
at $P_{i}$ $i=1,...,r$ ($2r$ linear conditions). If $\gamma $ is singular
at, say, $P_{1},...,P_{h}$ then $\gamma $ must be reducible into an element $%
\gamma ^{\prime }\in |A-b_{1}f_{1}-b_{2}f_{2}-...-b_{h}f_{h}|$ and the union
of fibres $f_{1},f_{2},...,f_{h}$ with multiplicities $b_{1},b_{2},...,b_{h}$
(see also the proof of Proposition \ref{due}). Now $\gamma ^{\prime }$ is
smooth at any $P_{i}$ and, by our assumptions, $\gamma ^{\prime }$ has
tangent vector \underline{$q$}$_{i}$ at any $P_{i}.$ One can choose a set of
$h$ generic fibres $\{$ $f_{1}^{\prime },f_{2}^{\prime },...,f_{h}^{\prime
}\}\cap F=\emptyset $ so that $\gamma ^{\prime }+b_{1}f_{1}^{\prime
}+b_{2}f_{2}^{\prime }+...+b_{h}f_{h}^{\prime }\in |A|$ satisfies our
request; recall that $b_{1}f_{1}^{\prime }+b_{2}f_{2}^{\prime
}+...+b_{h}f_{h}^{\prime }$ is linearly equivalent to $%
b_{1}f_{1}+b_{2}f_{2}+...+b_{h}f_{h}$.

Now let us assume that $\Phi $ is not empty and let us proceed as in the
previous case. If $\gamma ^{\prime }$ cuts every fibre of $\Phi $
tranversally we are done. Otherwise $\gamma ^{\prime }$ must be reducible
into an element $\gamma ^{\prime \prime }\in |A-b_{1}\varphi
_{1}-b_{2}\varphi _{2}-...-b_{h}\varphi _{h}|$ and the union of fibres $%
\varphi _{1},\varphi _{2},...,\varphi _{h}$ with multiplicities $%
b_{1},b_{2},...,b_{h}$. One can then argue as in the previous case.
\end{proof}

\begin{definition}
Let $W$\ be a $0$-dimensional reduced scheme on a ruled surface $Y.$\ For
any fixed fibre $\overline{\mathit{f}}$\ of $Y$\ we can compute $length(%
\mathcal{O}_{Y}(\overline{\mathit{f}})\otimes \mathcal{O}_{W}).$\ Let $%
lm_{Y}(W)$\ be the maximun of these lengths as $\overline{\mathit{f}}$\
varies among the fibre of $Y.$
\end{definition}

We can now prove prove the main Theorem of this section.

\begin{theorem}
\label{Valma}Let $Y=\mathbf{F}_{e}$ and let $L=a_{l}C_{0}+b_{l}f$ and $%
M=a_{m}C_{0}+b_{m}f$ be line bundles over $Y.$ Let $\mathcal{E}$ be a rank $%
2 $ vector bundle over $Y$ such that there exists a non trivial exact
sequence $0\rightarrow L\rightarrow \mathcal{E}\rightarrow M\otimes \mathcal{%
I}_{W}\rightarrow 0$ where $\mathcal{I}_{W}$ is the ideal sheaf of a $0$%
-dimensional reduced subscheme $W\subset Y$ of length $w.$ Let $X=$ $\Bbb{P}(%
\mathcal{E)}$ and let $T$ be the tautological line bundle on $X.$

Assume that there exist integers $x\geq e+2$ and $z\geq 1$ such that:

1) $L(C_{0}+xf)>0,$ $M(C_{0}+xf)>2$

2) $Lf>0,$ $Mf>lm_{Y}(W)$

3) $h^{1}(Y,L-zC_{0}-zxf)=0,$ $h^{1}(Y,M-zC_{0}-zxf)=0$

4) $h^{1}(Y,L-zC_{0}-zxf-f)=0,$ $h^{1}(Y,M-zC_{0}-zxf-f)=0$

5) the support of $W$ is in general position with respect to the linear
systems $|M-z(C_{0}+xf)|$ and $|M-z(C_{0}+xf)-f|$ (i.e. the following
natural maps $H^{0}(Y,M-z(C_{0}+xf))\rightarrow
H^{0}(Y,(M-z(C_{0}+xf))\otimes \mathcal{O}_{W})$ and $%
H^{0}(Y,M-z(C_{0}+xf)-f)\rightarrow H^{0}(Y,(M-z(C_{0}+xf)-f)\otimes
\mathcal{O}_{W})$ are surjective)

6) $(L+M)(C_{0}+xf)\geq 2(z-1)(2x-e).$

Then $T$ is very ample.
\end{theorem}

\begin{proof}
Let us consider the linear system $|C_{0}+xf|$ on $Y$ and notice that it is
very ample by the assumption $x\geq e+2$ and \cite{h}, pag.169. Moreover
there are no $|C_{0}+xf|-lines$ by Lemma \ref{nolines} (other than fibres).
Let $\gamma \in |C_{0}+xf|$ be a smooth rational curve passing through at
most $2$ points of $W.$ Let us consider the restriction $\mathcal{E}
_{|\gamma }$. By assumption 1) $\mathcal{E}_{|\gamma }$ is ample, hence very
ample because $\gamma $ is a smooth rational curve.

A similar argument shows that, for any fibre $\mathit{f}$ of the ruling of $%
Y,$ the restriction $\mathcal{E}_{|\mathit{f}}$ is also very ample, by
assumption 2).

Let us consider the following exact sequences:\newline
$0\rightarrow L-z\gamma \rightarrow \mathcal{E}\otimes \mathcal{O}
_{Y}(-z\gamma )\rightarrow (M-z\gamma )\otimes \mathcal{I}_{W}\rightarrow 0,$%
\newline
$0\rightarrow (M-z\gamma )\otimes \mathcal{I}_{W}\rightarrow M-z\gamma
\rightarrow (M-z\gamma )\otimes \mathcal{O}_{W}\rightarrow 0.$

By using assumptions 3) and 5) it is easy to see that $h^{1}(Y,\mathcal{E}%
\otimes \mathcal{O}_{Y}(-z\gamma ))=0.$

Let us consider the following exact sequences:\newline
$0\rightarrow L-z\gamma -f\rightarrow \mathcal{E}\otimes \mathcal{O}
_{Y}(-z\gamma -f)\rightarrow (M-z\gamma -f)\otimes \mathcal{I}%
_{W}\rightarrow 0$, $0\rightarrow (M-z\gamma -f)\otimes \mathcal{I}
_{W}\rightarrow M-z\gamma -f\rightarrow (M-z\gamma -f)\otimes \mathcal{O}%
_{W}\rightarrow 0.$

By using assumptions 4) and 5) it is easy to see that $h^{1}(Y,\mathcal{E}%
\otimes \mathcal{O}_{Y}(-z\gamma -f))=0.$

Given a divisor $D\equiv $ $z\pi ^{*}\gamma $ on $X$ and the exact sequence
\newline
$0\rightarrow T-D\rightarrow T\rightarrow T\otimes \mathcal{O}%
_{D}\rightarrow 0,$ the above vanishings give that the natural map
\begin{equation}
|T|\rightarrow |T_{|D}|\to 0  \label{surjectivityD}
\end{equation}
is surjective. A similar argument shows that if $D\equiv z\pi ^{*}\gamma
+\pi ^{*}f,$ the map in (\ref{surjectivityD}) is surjective too.

To show that $T$ is very ample we have to prove that $|T|$ separates all
zero schemes $\xi \subset X^{[2]}.$ The proof is divided into two cases,
according to the nature of $\xi .$ When $\xi $ is not reduced, and $\mathrm{%
\ Supp}(\xi )=\{P\},$ we denote by $\underline{q}\in \Bbb{T}_{P}(X)$ the
tangent direction specified by $\xi .$

\textbf{Case 1:} $P\in \mathrm{Supp}(\xi )$ and $\xi \nsubseteq F$ (scheme
theoretically), where $F=\Bbb{P}(\mathcal{E}_{|f})$ and $f=\rho ^{-1}(\rho
(\pi (P))$.

If $\mathrm{Supp}(\xi )=\{P,Q\},$ as $|C_{0}+xf|$ is very ample, \cite{bs}
Theorem 1.7.9, and Lemma \ref{nolines} imply that there exists a smooth
rational curve $\gamma _{1}\in |C_{0}+xf|$ containing $\pi (P)$ and $\pi
(Q). $ If $\xi $ is not reduced, again very ampleness of $|C_{0}+xf|,$ \cite
{bs} Theorem 1.7.9, and Lemma \ref{nolines} imply that there exists a smooth
rational curve $\gamma _{1}\in |C_{0}+xf|$ passing through $\pi (P),$ with $%
\underline{q}\in \Bbb{T}_{P}(\Gamma _{1}),$ where $\Gamma _{1}=\Bbb{P}(%
\mathcal{E}_{|\gamma _{1}}).$ By choosing suitably $\gamma _{1}$ we can
assume that $\gamma _{1}$ contains at most two of the points of $W$ ($\pi
(P) $ and $\pi (Q)$ if it is the case): recall that there are no curves,
other than fibres, embedded as lines by $|C_{0}+xf|$ thanks to Lemma \ref
{nolines}. Choose another smooth rational curve $\gamma _{2}\in |C_{0}+xf|$
such that $\gamma _{2}$ intersects $\gamma _{1}$ transversely at $%
(C_{0}+xf)^{2}=2x-e$ points all different from $\pi (\mathrm{Supp}(\xi ))$
and not belonging to $W.$ Now choose another smooth rational curve $\gamma
_{3}\in |C_{0}+xf|$ such that $\gamma _{3}$ and $\Gamma _{3}$ have the same
above properties with respect to $\gamma _{1},\gamma _{2,}\Gamma _{1},\Gamma
_{2},W$ and moreover $\bigcap_{i=1}^{3}\gamma _{i}=\emptyset .$ Iterate the
process, until $i=z.$ Note that the suitable choice of the $\gamma _{i}$ is
possible simply because $|C_{0}+xf|$ is very ample on $Y.$

Let $\Gamma _{i}$ also denote the numerical classes of $\Gamma _{i}$ in $%
\text{\textrm{Num}}(X).$ If $z=1$ we can separate $\xi $ on $\Gamma _{1}$
because $T_{|\Gamma _{1}}$ is very ample, then we are done by lifting the
separating sections with the map in (6). If $z\geq 2$ we need to apply Lemma
\ref{punti-tg} to $T_{|\Gamma _{i}}$ when $i\geq 2.$ Let us remark that $%
h^{0}(\Gamma _{i},T_{|\Gamma _{i}})=$ $h^{0}(\Gamma _{2},T_{|\Gamma
_{2}})=(L+M)\gamma _{2}+2=(L+M)(C_{0}+xf)+2$ for $i=2,...,z,$ so that
assumption 6) implies that $h^{0}(\Gamma _{i},T_{|\Gamma _{i}})\geq
2(z-1)(2x-e)$ for any $i=2,...,z.$ Noticing that $\Gamma _{i}\cap \Gamma
_{j} $ is the disjoint union of $2x-e$ fibres of $\pi ,$ we can proceed as
follows: $T_{|\Gamma _{1}}$ is very ample, so we can take $s_{1}^{1}\in
|T_{|\Gamma _{1}}|$ and $s_{1}^{2}\in |T_{|\Gamma _{1}}|$ separating $\xi $
on $\Gamma _{1}$. We can also choose $s_{1}^{1}$ and $s_{1}^{2}$ such that
they intersect transversely all the fibres $\Gamma _{1}\cap \Gamma _{j}$ $%
j=2,...,z.$ As $h^{0}(\Gamma _{2},T_{|\Gamma _{2}})\geq 2(2x-e),$ by Lemma
\ref{punti-tg}, we can choose $s_{2}^{1}\in |T_{|\Gamma _{2}}|$ and $%
s_{2}^{2}\in |T_{|\Gamma _{2}}|$ such that $s_{1}^{k}\cap \Gamma
_{2}=s_{2}^{k}\cap \Gamma _{1}$, for $k=1,2,$ and such that they intersect
transversely all the fibres $\Gamma _{2}\cap \Gamma _{j}$ $j=3,...,z.$ As $%
h^{0}(\Gamma _{3},T_{|\Gamma _{3}})\geq 4(2x-e),$ by Lemma \ref{punti-tg},
we can choose $s_{3}^{1}\in |T_{|\Gamma _{3}}|$ and $s_{3}^{2}\in
|T_{|\Gamma _{3}}|$ such that $s_{3}^{k}\cap \Gamma _{2}=s_{2}^{k}\cap
\Gamma _{3}$, and $s_{3}^{k}\cap \Gamma _{1}=s_{1}^{k}\cap \Gamma _{3}$ for $%
k=1,2,$ and such that they intersect transversely all the fibres $\Gamma
_{3}\cap \Gamma _{j}$ $j=4,...,z.$ And so on. At the end we get, for each $%
k=1,2$, a set of $z$ elements $s_{i}^{k}\in |T_{|\Gamma _{i}}|$ $i=1,..,z$
such that for any $i,j\in \{1,...,z\},$ $i\neq j,$ $s_{i}^{k}\cap \Gamma
_{j}=s_{j}^{k}\cap \Gamma _{i}$ is a reduced zero-subscheme, and $%
\{s_{1}^{1},s_{1}^{2}\}$ separate $\xi $ on $\Gamma _{1}.$ Thus, for each $%
k=1,2,$ the $z$-tuples $\{s_{i}^{k}\in |T_{|\Gamma _{i}}|$ $i=1,...,z\}$
give rise to elements $\overline{s}^{k}$ of $|T_{|(\Gamma _{1}\cup ....\cup
\Gamma _{z})}|$ $=|T_{|(\Gamma _{1}+...+\Gamma _{z})}|$ which still separate
$\xi $ on $\Gamma _{1}\cup ...\cup \Gamma _{z}.$ As $\Gamma _{1}+...+\Gamma
_{z}\equiv z\pi ^{*}\gamma ,$ (\ref{surjectivityD}) gives elements $s^{k}$
of $|T|$ that separate $\xi $ on $X.$

\textbf{Case 2:} $P\in \mathrm{Supp}(\xi )$ and $\xi \subset F$ (scheme
theoretically), where $F=\Bbb{P}(\mathcal{E}_{|f})$ and $f=\rho ^{-1}(\rho
(\pi (P))$.

Let $F$ also denote the numerical class of $\Bbb{P}(\mathcal{E}_{|f})$ in $%
\text{\textrm{Num}}(X).$ As in the previous case, let us choose $z$ smooth
rational curves $\gamma _{1},...,\gamma _{z}\in |C_{0}+xf|$ such that they
intersect each other transversely at $(C_{0}+xf)^{2}=2x-e$ points, each of
them intersects $f$ transversely at one point different from $\pi (\mathrm{%
Supp}(\xi )),$ not belonging to $W,$ and such that any $3$-tuple of curves
in $\{\gamma _{1},...,\gamma _{z},\mathit{f}\}$ has empty intersection. Note
that all the suitables choices ar possible because $|C_{0}+xf|$ is very
ample.

Let $\Gamma _{i}$ also denote the numerical class of $\Gamma _{i}=\Bbb{P}(%
\mathcal{E}_{|\gamma _{i}})$ in $\text{\textrm{Num}}(X).$ We get the same
results about $h^{0}(\Gamma _{i},T_{|\Gamma _{i}})$ as in Case 1 and we have
$h^{0}(\Gamma _{i},T_{|\Gamma _{i}})=(L+M)(C_{0}+xf)+2$ for any $i=1,...,z.$

As $T_{|F}$ is very ample we can take $s_{0}^{1}\in |T_{|F}|$ and $%
s_{0}^{2}\in |T_{|F}|$ separating $\xi $ on $F$. We can also choose $%
s_{0}^{1}$ and $s_{0}^{2}$ such that they intersect transversely all the
fibres $F\cap \Gamma _{j}$ $j=1,...,z.$ Then we can choose $s_{1}^{1}\in
|T_{|\Gamma _{1}}|$ and $s_{1}^{2}\in |T_{|\Gamma _{1}}|$ such that $%
s_{1}^{k}\cap F=s_{0}^{k}\cap \Gamma _{1}$, for $k=1,2,$ and such that they
intersect transversely all the fibres $\Gamma _{1}\cap \Gamma _{j}$ $%
j=2,...,z.$ As $h^{0}(\Gamma _{2},T_{|\Gamma _{2}})\geq 2[(2x-e)+1],$ by
Lemma \ref{punti-tg}, we can choose $s_{2}^{1}\in |T_{|\Gamma _{2}}|$ and $%
s_{2}^{2}\in |T_{|\Gamma _{2}}|$ such that $s_{1}^{k}\cap \Gamma
_{2}=s_{2}^{k}\cap \Gamma _{1}$, $s_{2}^{k}\cap F=s_{0}^{k}\cap \Gamma _{2}$
for $k=1,2,$ and such that they intersect transversely all the fibres $%
\Gamma _{3}\cap \Gamma _{j}$ $j=3,...,z.$ and so on. At the end, by
recalling that assumption 6) gives $h^{0}(\Gamma _{i},T_{|\Gamma _{i}})\geq
2[(z-1)(2x-e)+1],$ for any $i=1,...,z,$ we get a $(z+1)$-tuples $%
\{s_{i}^{k}, $ $i=0,1,...,z\}$ for $k=1,2,$ giving rise to two elements $%
\overline{s}^{k}\in |T_{|(F\cup \Gamma _{1}\cup ...\cup \Gamma _{z})}|$ $%
=|T_{|(F+\Gamma _{1}+...+\Gamma _{z}}|$ that separate $\xi $ on $F\cup
\Gamma _{1}\cup ...\cup \Gamma _{z}.$ As $\Gamma _{1}+...+\Gamma
_{z}+F\equiv z\pi ^{*}\gamma +\pi ^{*}f,$ (\ref{surjectivityD}) gives
elements $s^{k}\in |T|$ that separate $\xi $ on $X.$
\end{proof}

\begin{remark}
\label{remnum} Note that the previous criterion is purely numerical because
condition 5) can be translated into a vanishing condition as we have seen in
the proof of Theorem \ref{Valma}.
\end{remark}

To show the large validity of Theorem \ref{Valma} we give the following
example.

\textbf{Example:} let us consider on $Y=\mathbf{F}_{e}$ two line bundles $%
L\equiv aC_{0}+b_{l}f$ and $M\equiv (a+2)C_{0}+b_{m}f$ with $a>0.$ Let us
consider $K_{Y}\equiv -2C_{0}-(2+e)f$ so that $K_{Y}+M-L\equiv
(b_{m}-b_{l}-e-2)f.$ Let us assume $b_{m}-b_{l}-e-2>0$ and let us choose $%
W=\{$two distinct points on a fixed fibre $\overline{f}\},$ hence $w=2.$ In
this way we can apply Griffiths-Harris theorem on the existence of rank $2$
vector bundles on surfaces (because every section of $|K_{Y}+M-L|$ passing
through a point of $W,$ passes also through the other point, see \cite{dl},
Theorem 3.13 and its proof) and we get an exact sequence as the following: $%
0\rightarrow \mathcal{O}_{Y}\rightarrow \mathcal{E}^{\prime }\rightarrow
(M-L)\otimes \mathcal{I}_{W}\rightarrow 0.$

By tensorizing with $L$ we get our vector bundle $\mathcal{E}$:

\begin{center}
$0\rightarrow L\rightarrow \mathcal{E}\rightarrow M\otimes \mathcal{I}%
_{W}\rightarrow 0.$
\end{center}

Let us fix $x=e+2$ and $z=a+1$ and let us write down all necessary numerical
conditions to satisfy the assumptions of Theorem \ref{Valma}.

1) $2a+b_{l}>0$ and $2a+4+b_{m}>2$.

2) $a>0.$

3) $b_{m}-(a+1)(e+2)\geq e-1$.

4) $b_{m}-(a+1)(e+2)-1\geq e-1$.

5) $b_{m}-(a+1)(e+2)-1\geq e$ and $b_{m}-(a+1)(e+2)-2\geq e$.

6) $b_{l}+b_{m}\geq 4a+2ae-4.$

By looking at all required conditions in theorem \ref{Valma} one can show
that, for fixed $e,a>0,b_{l}>-2a,$ all assumptions are satisfied for $%
b_{m}>>0,$ thus obtaining the very ampleness of a large class vector bundles
$\mathcal{E}.$ These bundles give rise to $3$-folds scrolls embedded with
fairly large degree (see Remark \ref{deg11g6onP2} and Section 7). For
instance, if $d:=\deg (X),$ we have $d=c_{1}(\mathcal{E})^{2}-c_{2}(\mathcal{%
E})=(3a+4)(b_{m}+b_{l})-2b_{l}-e(3a^{2}+6a+4)-2,$ and, if we put $e=a=1,$ $%
b_{l}=-1,b_{m}=9$, we get $d=43.$

\section{A criterion for ruled surfaces on elliptic curves}

In this section we want to give some very ampleness criterion when the
surface $Y$ is a ruled surface on an elliptic curve, i.e. $Y=\Bbb{P}(%
\mathcal{F}),$ where $\mathcal{F}$ is a normalized rank $2$ vector bundle
over a smooth elliptic curve $C$ with invariant $e.$ Normalized means that $%
h^{0}(C,\mathcal{F})\neq 0,$ but $h^{0}(C,\mathcal{F}\otimes \mathcal{L})=0$
for any linear bundle $\mathcal{L}$ on $C$ of negative degree. As usual, let
$\rho :Y\rightarrow C$ be the natural map.

To get our criterion, in this case we will need stronger assumptions. Let us
begin with the following.

\begin{lemma}
\label{punti-tge} Let $Y$ be an elliptic ruled surface as above over an
elliptic curve $C.$ Let $A\equiv C_{0}+bf$ a very ample divisor on $Y$
(where $C_{0}$ and $f$ are, respectively, the numerical classes of the
tautological line bundle and of the fibre). Let $P_{1},...,P_{r}$ be $r\geq
2 $ distinct points on $Y$ such that no two of them lie on the same fibre.
Let $F:=$ $\{f_{P_{1}},f_{P_{2}},...,f_{P_{r}}\}$ where $f_{P_{i}}=\rho
^{*}\rho (P_{i}).$ If $h^{0}(Y,A)\geq 2r+2$ then there exists at least an
element $\gamma \in |A|$ passing smoothly through $P_{1},...,P_{r}$,
intersecting any $f_{P_{i}}$ tranversely and intersecting transversely any
other finite set of fibres $\Phi :=$ $\{\varphi _{j}\}$ with $F\cap \Phi
=\emptyset .$
\end{lemma}

\begin{proof}
Firstly, let us assume that $\Phi $ is empty. For any point $P_{i}$ let us
choose a vector \underline{$q$}$_{i}\in \Bbb{T}_{P_{i}}(Y)$ pointing towards
a direction different from that of $f_{P_{i}}$. By assumption, there exists
at least a pencil of elements $\gamma \in |A|$ passing through $%
P_{1},...,P_{r}$ and having tangent direction \underline{$q$}$_{i}$ at $%
P_{i} $ $i=1,...,r$ ($2r$ linear conditions). If there exists at least an
element $\gamma \in |A|$ passing smoothly through $P_{1},...,P_{r}$ and
cutting any $f_{P_{i}}$ tranversally we are done. Otherwise every $\gamma $
is singular at some points, say, $P_{1},...,P_{s}$ and therefore every $%
\gamma $ must be reducible into an element $\gamma ^{\prime }\in
|A-b_{1}f_{P_{1}}-b_{2}f_{P_{2}}-...-b_{s}f_{P_{s}}|$ and the union of
fibres $f_{P_{1}},f_{P_{2}},...,f_{P_{s}}$ with multiplicities $%
b_{1},b_{2},...,b_{s}$ (see also the proof of Proposition \ref{due}). Now
the generic $\gamma ^{\prime }$ is smooth at any $P_{i}$ and, by our
assumptions, $\gamma ^{\prime }$ has tangent vector \underline{$q$}$_{i}$ at
any $P_{i}.$

Let $\rho :Y\rightarrow C$ be the natural projection. If $%
b:=b_{1}+b_{2}+...+b_{s}\geq 3$, we can choose other $b$ generic points $%
H_{1},...,H_{b}\in C$ (not necessarily distinct) such that $%
f_{H_{1}}+f_{H_{2}}+...+f_{H_{b}}$ is linearly equivalent to $%
b_{1}f_{P_{1}}+b_{2}f_{P_{2}}+...+b_{s}f_{P_{s}}$. This is possible because $%
b_{1}\rho (P_{1})+b_{2}\rho (P_{2})+...+b_{s}\rho (P_{s})$ is a very ample
divisor on $C,$ so that it suffices to choose a generic divisor $%
H_{1}+...+H_{b}$ (disjoint with $b_{1}\rho (P_{1})+b_{2}\rho
(P_{2})+...+b_{s}\rho (P_{s})$) in the linear system $|b_{1}\rho
(P_{1})+b_{2}\rho (P_{2})+...+b_{s}\rho (P_{s})|.$ Now we can consider an
element $\gamma ^{\prime }+f_{H_{1}}+f_{H_{2}}+...+f_{H_{b}}\in |A|$ and we
are done.

If $b:=b_{1}+b_{2}+...+b_{s}=2$ we can choose other $2$ generic points $%
H_{1},H_{2}\in C$ (not necessarily distinct) such that $f_{H_{1}}+f_{H_{2}}$
is linearly equivalent to $b_{1}f_{P_{1}}+b_{2}f_{P_{2}}+...+b_{s}f_{P_{s}}$
. This is possible because $b_{1}\rho (P_{1})+b_{2}\rho
(P_{2})+...+b_{s}\rho (P_{s})$ is a degree $2$ effective divisor on $C,$
hence $|b_{1}\rho (P_{1})+b_{2}\rho (P_{2})+...+b_{s}\rho (P_{s})|$ is a one
dimensional linear system without base points and it suffices to choose a
generic divisor $H_{1}+H_{2}$ (disjoint with $b_{1}\rho (P_{1})+b_{2}\rho
(P_{2})+...+b_{s}\rho (P_{s})$) in it. Now we can consider an element $%
\gamma ^{\prime }+f_{H_{1}}+f_{H_{2}}\in |A|$ and we are done.

If $b:=b_{1}+b_{2}+...+b_{s}=1$, say $b_{1}=1$ and $b_{2}=...=b_{s}=0,$
every $\gamma $ must be reducible into an element $\gamma ^{\prime }\in
|A-f_{P_{1}}|$ and the fibre $f_{P_{1}}.$ Because there exists at least a
pencil of elements $\gamma \in |A|$ passing through $P_{1},...,P_{r}$ and
having tangent direction \underline{$q$}$_{i}$ at $P_{i}$ $i=1,...,r,$ we
can choose another point $P^{\prime }\in f_{P_{2}},$ $P^{\prime }\neq P_{2}$
(recall that $r\geq 2$) and we can find at least an element $\overline{%
\gamma }\in |A|$ passing through $P_{1},...,P_{r},P^{\prime }$ and having
tangent direction \underline{$q$}$_{i}$ at $P_{i}$ $i=1,...,r.$ We have that
$\overline{\gamma }$ is reducible into an element $\overline{\gamma ^{\prime
}}\in |A-f_{P_{1}}-f_{P_{2}}|$ and the two fibres $f_{P_{1}},f_{P_{2}}$. As
above, $\overline{\gamma ^{\prime }}$ is smooth at any $P_{i}$ and, by our
assumptions, $\overline{\gamma ^{\prime }}$ has tangent vector \underline{$q$%
}$_{i}$ at any $P_{i},$ moreover, as above, we can choose a generic divisor $%
H_{1}+H_{2}$ (disjoint with $\rho (P_{1})+\rho (P_{2})$) such that $%
f_{H_{1}}+f_{H_{2}}$ is linearly equivalent to $f_{P_{1}}+f_{P_{2}}$. Now we
can consider the element $\overline{\gamma ^{\prime }}+f_{H_{1}}+f_{H_{2}}%
\in |A|$ and we are done.

Now let us assume that $\Phi $ is not empty, with $j=1,2,...,p,$ and let us
proceed as in the previous case. If $\gamma ^{\prime }$ (or $\overline{
\gamma ^{\prime }}$ ) intersects every fibre of $\Phi $ tranversally we are
done. Otherwise $\gamma ^{\prime }$ (or $\overline{\gamma ^{\prime }}$ )
must be reducible into an element $\gamma ^{\prime \prime }\in
|A-b_{1}\varphi _{1}-b_{2}\varphi _{2}-...-b_{p}\varphi _{p}|$ and the union
of fibres $\varphi _{1},\varphi _{2},...,\varphi _{p}$ with multiplicities $%
b_{1},b_{2},...,b_{p}$. One can then argue as in the previous case. Note
that, as all our choices are made by using generic points $H_{i}$ on $C,$ we
can avoid any fixed set of points on $C$.
\end{proof}

Now let us return to the surface $Y=\Bbb{P}(\mathcal{F}).$ It is well known
that, if $\mathcal{F}$ is indecomposable, hence semistable, then $0\leq
-e:=\deg (\mathcal{F})=\deg [c_{1}(\mathcal{F})]\leq 1;$ if $\mathcal{F}$ is
decomposable then $\mathcal{F}=\mathcal{O}_{C}\oplus \overline{\mathcal{L}}$
where $\overline{\mathcal{L}}$ is a line bundle of degree $-e\leq 0$; (see%
\cite{h}, V, Theorem 2.12 and 2.15). Moreover if we compute the invariant $%
\mu ^{-}(\mathcal{F})$ (see \cite{bu} for the definition) we have that $\mu
^{-}(\mathcal{F})=-\frac{e}{2}$ in the first case and $\mu ^{-}(\mathcal{F}%
)=-e$ in the second case.

We have the following version of Theorem \ref{Valma} for elliptic ruled
surfaces:

\begin{theorem}
\label{Valmae} Let $Y=\Bbb{P}(\mathcal{F})$ be a surface as above. Let $%
L\equiv a_{l}C_{0}+b_{l}f$, $M\equiv a_{m}C_{0}+b_{m}f$ be line bundles over
$Y.$ Let $\mathcal{E}$ be a rank $2$ vector bundle over $Y$ such that there
exists an exact sequence $0\rightarrow L\rightarrow \mathcal{E}\rightarrow
M\otimes \mathcal{I}_{W}\rightarrow 0$ where $\mathcal{I}_{W}$ is the ideal
sheaf of a $0$-dimensional reduced scheme $W\subset Y$ of length $w.$ Let $X=%
\Bbb{P}(\mathcal{E})$ and let $T$ be the tautological line bundle on $X.$

Assume that there exist integers $x$ and $z\geq 1$ such that:

\textit{0)} $x+\mu ^{-}(\mathcal{F})\geq 3$ $(or$ $x+\mu ^{-}(\mathcal{F})>1$
$if$ $\mathcal{F}$ $has$ $degree$ $1)$

\textit{1)} $\min \{xa_{l}+b_{l}-a_{l}e,xa_{m}+b_{m}-a_{m}e-2\}\geq 3$

\textit{2)} $Lf>0,$ $Mf>lm_{Y}(W)$

\textit{3)} $h^{1}(Y,L-zC_{0}-zxf)=0,$ $h^{1}(Y,M-zC_{0}-zxf)=0$

\textit{4)} $h^{1}(Y,L-zC_{0}-zxf-f)=0,$ $h^{1}(Y,M-zC_{0}-zxf-f)=0$

\textit{5)} the support of $W$ is in general position with respect to the
linear systems $|M-z(C_{0}+xf)|$ and $|M-z(C_{0}+xf)-\rho ^{*}P|$ (i.e. the
following natural maps $H^{0}(Y,M-z(C_{0}+xf))\rightarrow
H^{0}(Y,(M-z(C_{0}+xf))\otimes \mathcal{O}_{W})$ and $H^{0}(Y,M-z(C_{0}+xf)-%
\rho ^{*}P)\rightarrow H^{0}(Y,(M-z(C_{0}+xf)-\rho ^{*}P)\otimes \mathcal{O}%
_{W})$ are surjective), where $P$ is any point of $C$

\textit{6)} $(a_{l}+a_{m})(x-e)+b_{l}+b_{m}\geq 2(z-1)[(2x-e)+1]+2.$

Then $T$ is very ample.
\end{theorem}

\begin{proof}
Let us recall (see \cite{m}) that a divisor on $Y,$ whose numerical class is
$\alpha C_{0}+\beta f,$ is ample if and only if $\alpha \geq 1,$ $\beta
+\alpha \mu ^{-}(\mathcal{F})>0.$ The divisor is very ample if $\beta
+\alpha \mu ^{-}(\mathcal{F})\geq 3$ or $\beta +\alpha \mu ^{-}(\mathcal{F}
)>1$ if $\mathcal{F}$ has degree $1$ (see \cite{abb}).

Let us consider a linear system $|C_{0}+\rho ^{*}\varkappa |$ with $\deg
(\varkappa ):=x.$ By\textit{\ }0) it is very ample. Let $\gamma $ be any
smooth elliptic curve in this system passing through at most $2$ points of $%
W.$ Let us consider the restriction $\mathcal{E}_{|\gamma }$. We get an
exact sequence: $0\rightarrow \mathcal{L}\rightarrow \mathcal{E}_{|\gamma
}\rightarrow \mathcal{M}\rightarrow 0$ where $\mathcal{L}$ and $\mathcal{M}$
are linear bundles on $\gamma $ of degrees $xa_{l}+b_{l}-a_{l}e+\varepsilon $
and $xa_{m}+b_{m}-a_{m}e-\varepsilon $ respectively, and $\varepsilon =0,1,2$
is the number of common points among $\gamma $ and $W.$ Recall that $\mu
^{-}(\mathcal{E}_{|\gamma })\geq \min \{xa_{l}+b_{l}-a_{l}e+\varepsilon
,xa_{m}+b_{m}-a_{m}e-\varepsilon \}$ (see \cite{bu}). By 1)\textit{\ }$\mu
^{-}(\mathcal{E}_{|\gamma })\geq 3,$ so that $\mathcal{E}_{|\gamma }$ is
very ample by Theorem 3.3 and Proposition 3.2 of \cite{abb}.

If we fix any fibre $\mathit{f}$ of $Y$ and we consider $\mathcal{E}_{|%
\mathit{f}}$, we get that $\mathcal{E}_{|\mathit{f}}$ is isomorphic to $%
\mathcal{O}_{\Bbb{P}^{1}}(h)\oplus \mathcal{O}_{\Bbb{P}^{1}}(k),$ with $%
h,k>0 $ by assumption 2), so that $\mathcal{E}_{|\mathit{f}}$ is very ample.

Let us consider the following exact sequences:\newline
$0\rightarrow L-z\gamma \rightarrow \mathcal{E}\otimes \mathcal{O}
_{Y}(-z\gamma )\rightarrow (M-z\gamma )\otimes \mathcal{I}_{W}\rightarrow 0,$%
\newline
$0\rightarrow (M-z\gamma )\otimes \mathcal{I}_{W}\rightarrow M-z\gamma
\rightarrow (M-z\gamma )\otimes \mathcal{O}_{W}\rightarrow 0.$

By using assumptions 3) and 5) it is easy to see that $h^{1}(Y,\mathcal{E}%
\otimes \mathcal{O}_{Y}(-z\gamma ))=0.$

Let us consider the following exact sequences, for any point $P\in C$:%
\newline
$0\rightarrow L-z\gamma -\rho ^{*}P\rightarrow \mathcal{E}\otimes \mathcal{O}%
_{Y}(-z\gamma -\rho ^{*}P)\rightarrow (M-z\gamma -\rho ^{*}P)\otimes
\mathcal{I}_{W}\rightarrow 0,$\newline
$0\rightarrow (M-z\gamma -\rho ^{*}P)\otimes \mathcal{I}_{W}\rightarrow
M-z\gamma -\rho ^{*}P\rightarrow (M-z\gamma -\rho ^{*}P)\otimes \mathcal{O}
_{W}\rightarrow 0.$

By using assumptions 4) and 5) it is easy to see that $h^{1}(Y,\mathcal{E}
\otimes \mathcal{O}_{Y}(-z\gamma -\rho ^{*}P))=0.$

Let us consider any effective divisor $D\equiv $ $z\pi ^{*}\gamma $ on $X$
and the exact sequence: $0\rightarrow T-D\rightarrow T\rightarrow T\otimes
\mathcal{O}_{D}\rightarrow 0.$ The natural map $|T|\rightarrow |T_{|D}|$ is
surjective as $h^{1}(X,T-D)=h^{1}(Y,\mathcal{E}\otimes \mathcal{O}%
_{Y}(-z\gamma ))=0$ by the above vanishing. A similar argument shows that if
$D\equiv $ $z\pi ^{*}\gamma +\pi ^{*}\rho ^{*}P,$ the natural map is
surjective too, for any point $P\in C.$

Now we proceed as in the proof of Theorem \ref{Valma} after recalling that,
in our assumptions, on $Y$ there are no $|C_{0}+\pi ^{*}\varkappa |$-lines
except for the ruling by Lemma \ref{nolines}.

Now we can repeat the proof of Theorem \ref{Valma}, more or less verbatim.
We have to use Lemma \ref{punti-tge} with $r=(C_{0}+\pi ^{*}\varkappa
)^{2}=(2x-e)$ (or $r=(C_{0}+\pi ^{*}\varkappa )^{2}+1=(2x-e)+1$) instead of
Lemma \ref{punti-tg}. Note that, in any case, assumption 0) implies that $%
r\geq 2$. The only difference is that here $h^{0}(\Gamma _{i},T_{|\Gamma
_{i}})=$ $h^{0}(\Gamma _{2},T_{|\Gamma _{2}})=(a_{l}+a_{m})(x-e)+b_{l}+b_{m}$
for any $i\geq 2,$ so that the last condition has to be: $%
(a_{l}+a_{m})(x-e)+b_{l}+b_{m}\geq 2(z-1)r+2$ in all cases, i.e. we need
assumption 6)\textit{.}
\end{proof}

\section{Very ampleness through general position}

In this section we consider some very special rank $2$ vector bundles over $%
Y=\mathbf{F}_{1}$. The aim of Section 6 will be evident in Section 7.

\begin{definition}
\label{defposiz} Let $W^{\prime }$\ be a $0$-dimensional scheme of lenght $%
w^{\prime }$\ in $\Bbb{P}^{2}$\ consisting of $w^{\prime }$\ simple points.
These points are said to be in general position on $\Bbb{P}^{2}$\ if, for
any positive integer $k,$\ $h^{0}(\Bbb{P}^{2},\mathcal{I}_{W^{\prime
}}(k))=\max \{h^{0}(\Bbb{P}^{2},\mathcal{O}_{\Bbb{P}^{2}}(k))-w^{\prime
},0\}=\max \{\binom{k+2}{2}-w^{\prime },0\}.$\ Let $s:Y\rightarrow $\ $\Bbb{P%
}^{2}$\ be the blow up of $\Bbb{P}^{2}$\ at one point $P_{0},$\ and let $%
W\subset Y$\ be a $0$-dimensional scheme consisting of $w$\ simple points
none of which lie on the exceptional divisor. These points are said to be in
general position on $Y$\ if $P_{0}\cup s(W)$\ consists of $w^{\prime }=w+1$\
distinct points in general position on $\Bbb{P}^{2}$.
\end{definition}

Definition \ref{defposiz} can be reinterpreted in terms of cohomological
vanishing.

\begin{lemma}
\label{lemgenpos} Let $s:Y\rightarrow $ $\Bbb{P}^{2}$ be the blow up of $%
\Bbb{P}^{2}$ at one point $P_{0}.$ Let $W$ be a set of $w$ distinct points
on $Y$ in general position according to definition \ref{defposiz}; let $D$ $%
\equiv $ $aC_{0}+bf$ be a divisor on $Y$ such that $a\geq 0,$ $b\geq 1$ and $%
b\geq a.$ Let us assume that $h^{0}(Y,D)\geq w,$ then $h^{1}(Y,D\otimes
\mathcal{I}_{W})=0.$
\end{lemma}

\begin{proof}
Let us consider the exact sequence: $0\rightarrow D\otimes \mathcal{I}
_{W}\rightarrow D\rightarrow D\otimes \mathcal{O}_{W}\rightarrow 0.$

As $H^{1}(Y,D)=0$ we get our vanishing if (ad only if) the restriction map $%
r:H^{0}(Y,D)\rightarrow H^{0}(W,D\otimes \mathcal{O}_{W})$ is surjective.
Let $l$ be the pull back of the generator of \textrm{Pic}$(\Bbb{P}^{2})$ and
let $l_{0}$ be the exceptional divisor for $s,$ then $D=bl-(b-a)l_{0}$ and,
by assumptions, $D$ is an effective divisor. Hence $r$ is surjective if and
only if the $0$-dimensional scheme $s(W)$ imposes independent conditions on
plane curves of degree $b,$ having a point of multiplicity at least $b-a$ at
$P_{0}.$ This condition, using the same language, definitions and notation
introduced in \cite{cm}, pag. 192, is equivalent to the \textit{%
quasi-homogeneous} linear system $\mathcal{L}(b,b-a,w,1)$ not being \textit{%
\ special}. This in turn follows from Theorem 8.1 and Lemma 7.1 of \cite{cm}
as $P_{0}\cup s(W)$ are in general position in $\Bbb{P}^{2}$.
\end{proof}

Now, let $L\equiv C_{0}+(5-h)f$ and $M\equiv 2C_{0}+hf$ be two line bundles
on $Y$ with $h\geq 3.$ Let us fix an integer $y$ such that $-2\leq y\leq 4$
and let us choose a set $W$ of $w=h+y\geq 1$ distinct points on $Y$ in
general position according to definition \ref{defposiz}; in particular each
fibre contains at most one point of $W$. On our surface $K_{Y}\equiv
-2C_{0}-3f$ so that $|K_{Y}+M-L|=|-C_{0}+(2h-8)f|=\emptyset $ and therefore
we can apply Griffiths-Harris theorem on the existence of rank $2$ vector
bundles on surfaces (see \cite{dl}, Theorem 3.13 and its proof) and we get
an exact sequence as the following: $0\rightarrow \mathcal{O}_{Y}\rightarrow
\mathcal{E}_{y}^{\prime }\rightarrow (M-L)\otimes \mathcal{I}_{W}\rightarrow
0;$ by tensorizing it with $L$ we get rank $2$ vector bundles $\mathcal{E}%
_{y}$:

\begin{center}
$0\rightarrow L\rightarrow \mathcal{E}_{y}\rightarrow M\otimes \mathcal{I}%
_{W}\rightarrow 0$ $\qquad (*)$
\end{center}

such that $c_{1}(\mathcal{E}_{y})\equiv 3C_{0}+5f,$ $c_{2}(\mathcal{E}
_{y})=8+y.$ In this section we prove very amplenes results for $\mathcal{E}
_{y}$, with $-2\leq y\leq 3,$ for suitable sets $W$ of $h+y$ generic points
of $Y.$ These results are presented in Theorems \ref{teoy3} and \ref{teoy4}
after a number of preparatory Lemmas.

Firstly we have the following.

\begin{proposition}
\label{restriction} Let $\mathcal{E}_{y}$ be any vector bundle defined as
above by $(*)$. Let $\overline{f}$ be any fixed fibre of $Y.$ Let $\gamma $
be a smooth element of $|C_{0}+f|$ passing through at most two points of $W.$
Then:

$i)$ the restriction $\mathcal{E}_{y|\overline{f}}$ is isomorphic to $%
\mathcal{O}_{\Bbb{P}^{1}}(1)\oplus \mathcal{O}_{\Bbb{P}^{1}}(2),$ hence very
ample;

$ii)$ if $h=3$ then $h^{1}(Y,\mathcal{E}_{y}\otimes \mathcal{O}_{Y}(-%
\overline{f}))=0$ for $-2\leq y\leq 3;$

$iii)$ if $h=4$ then $h^{1}(Y,\mathcal{E}_{y}\otimes \mathcal{O}_{Y}(-%
\overline{f}))=0$ for $-2\leq y\leq 4;$

$iv)$ if $h=3$ the restriction $\mathcal{E}_{y|C_{0}}$ is very ample;

$v)$ if $h=4$ the restriction $\mathcal{E}_{y|C_{0}}$ is very ample or,
possibly, the rational map associated to the linear system of the
tautological divisor in $\Bbb{P}(\mathcal{E}_{y|C_{0}})$ is a birational
morphism, contracting only $C_{0}$ at a singular double point;

$vi)$ if $h=3$ or $h=4,$ $h^{1}(Y,\mathcal{E}_{y}$ $\otimes \mathcal{O}%
_{Y}(-C_{0}))=0;$

$vii)$ if $h=3$ or $h=4,$ the restriction $\mathcal{E}_{y|\gamma }$ is very
ample;

$viii)$ if $h=3$ then $h^{1}(Y,\mathcal{E}_{y}$ $\otimes \mathcal{O}%
_{Y}(-\gamma ))=0$ for $-2\leq y\leq 2;$

$ix)$ if $h=4$ then $h^{1}(Y,\mathcal{E}_{y}$ $\otimes \mathcal{O}%
_{Y}(-\gamma ))=0$ for $-2\leq y\leq 3.$
\end{proposition}

\begin{proof}
For $i)$ we restrict $(*)$ to $\overline{f}$ and we get the exact sequence:%
\newline
$0\rightarrow \mathcal{O}_{\Bbb{P}^{1}}(1+\varepsilon )\rightarrow \mathcal{E%
}_{y|\overline{f}}\rightarrow \mathcal{O}_{\Bbb{P}^{1}}(2-\varepsilon
)\rightarrow 0$\newline
where $\varepsilon =1$ or $\varepsilon =0$ according to whether $\overline{f}
$ contains one point of $W$ or not. Note that, in any case, $\mathcal{E}_{y|%
\overline{f}}$ $\simeq $ $\mathcal{O}_{\Bbb{P}^{1}}(1)\oplus \mathcal{O}_{%
\Bbb{P}^{1}}(2),$ hence very ample.

For $ii)$ and $iii)$ we tensorize $(*)$ by $\mathcal{O}_{Y}(-\overline{f})$
and we get:\newline
$0\rightarrow C_{0}+(4-h)f\rightarrow \mathcal{E}_{y}\otimes \mathcal{O}
_{Y}(-\overline{f})\rightarrow (2C_{0}+hf-\overline{f})\otimes \mathcal{I}%
_{W}\rightarrow 0.$\newline
In any case $h^{1}(Y,C_{0}+(4-h)f)$ $=0.$ Moreover $h^{1}(Y,(2C_{0}+hf-%
\overline{f})\otimes \mathcal{I}_{W})=0$, by Lemma \ref{lemgenpos}, if $%
h^{0}(Y,2C_{0}+(h-1)f)=3h-3\geq w=h+y.$ \newline
Hence $h^{1}(Y,\mathcal{E}_{y}\otimes \mathcal{O}_{Y}(-\overline{f}))=0$ for
$y\leq 3$ when $h=3$ and for $y\leq 4$ when $h=4.$

For $iv)$ and $v)$ we restrict $(*)$ to $C_{0}$ and we get the exact
sequence (recall that $W\cap C_{0}=\emptyset $)$:0\rightarrow \mathcal{O}_{%
\Bbb{P}^{1}}(4-h)\rightarrow \mathcal{E}_{y|C_{0}}\rightarrow \mathcal{O}_{%
\Bbb{P}^{1}}(h-2)\rightarrow 0.$ If $h=3$ then $\mathcal{E}_{y|C_{0}}$ is
very ample. If $h=4,$ we have two cases: $\mathcal{E}_{y|C_{0}}=\mathcal{O}_{%
\Bbb{P}^{1}}(1)\oplus \mathcal{O}_{\Bbb{P}^{1}}(1),$ very ample, or $%
\mathcal{E}_{y|C_{0}}=\mathcal{O}_{\Bbb{P}^{1}}\oplus \mathcal{O}_{\Bbb{P}
^{1}}(2);$ in this second case it is well known that $\Bbb{P}(\mathcal{E}%
_{y|C_{0}})\simeq \mathbf{F}_{2}$ and the tautological linear system gives
rise to a birational morphism sending $\mathbf{F}_{2}$ into a quadric cone
of rank $3$, singular only at its vertex, where the curve $C_{0}$ is
contracted.

For $vi)$ we tensorize $(*)$ by $\mathcal{O}_{Y}(-C_{0})$ and we get the
exact sequence (recall that $W\cap C_{0}=\emptyset $ ) $:0\rightarrow
(5-h)f\rightarrow \mathcal{E}_{y}\otimes \mathcal{O}_{Y}(-C_{0})\rightarrow
(C_{0}+hf)\otimes \mathcal{I}_{W}\rightarrow 0;$ now $h^{1}(Y,(5-h)f)$ $=0$
and $h^{1}(Y,(C_{0}+hf)\otimes \mathcal{I}_{W})=0$ by Lemma \ref{lemgenpos}
as $h^{0}(Y,C_{0}+hf)=2h+1\geq w=h+y$ .

For $vii)$ we restrict $(*)$ to $\gamma $ and we get the exact sequence:%
\newline
$0\rightarrow \mathcal{O}_{\Bbb{P}^{1}}(5-h+\varepsilon )\rightarrow
\mathcal{E}_{y|\gamma }\rightarrow \mathcal{O}_{\Bbb{P}^{1}}(h-\varepsilon
)\rightarrow 0$ where $\varepsilon =0,1,2$ according to whether $\gamma $
passes through $0,1,2$ points of $W.$ In any case $\mathcal{E}_{y|\gamma }$
is very ample.

For $viii)$ and $ix)$ we tensorize $(*)$ by $\mathcal{O}_{Y}(-\gamma )$ and
we get the exact sequence:\newline
$0\rightarrow (4-h)f\rightarrow \mathcal{E}_{y}\otimes \mathcal{O}
_{Y}(-\gamma )\rightarrow (2C_{0}+hf-\gamma )\otimes \mathcal{I}%
_{W}\rightarrow 0.$

In any case $h^{1}(Y,(4-h)f)$ $=0$. Moreover $h^{1}(Y,(2C_{0}+hf-\gamma
)\otimes \mathcal{I}_{W})=0$, by Lemma \ref{lemgenpos}, if $%
h^{0}(Y,2C_{0}+hf-\gamma )=2h-1\geq w=h+y.$ \newline
Hence $h^{1}(Y,\mathcal{E}_{y}\otimes \mathcal{O}_{Y}(-\gamma ))=0$ for $%
y\leq 2$ when $h=3$ and for $y\leq 3$ when $h=4.$
\end{proof}

Now we prove the following Lemmas.

\begin{lemma}
\label{jet} Let $\mathcal{E}_{y}$ be any vector bundle defined as above by $%
(*)$. Let $A\equiv C_{0}+xf$ be a divisor on $Y.$ Then $\mathcal{E}%
_{y}\otimes \mathcal{O}_{Y}(A)$ is very ample for $x>>0.$
\end{lemma}

\begin{proof}
Let us tensorize $(*)$ by $\mathcal{O}_{Y}(A).$ We get:

\begin{center}
$0\rightarrow L+A\rightarrow \mathcal{E}_{y}\otimes \mathcal{O}
_{Y}(A)\rightarrow (M+A)\otimes \mathcal{I}_{W}\rightarrow 0.$
\end{center}

By Proposition 4.2 of \cite{bds} it suffices to show that $\mathcal{E}%
_{y}\otimes \mathcal{O}_{Y}(A)$ is $1-jet$ ample and, by Proposition 4.1 of
\cite{bds} it suffices to show that $\mathcal{E}_{y}\otimes \mathcal{O}%
_{Y}(A-B)$ is generated by global sections where $B$ is a very ample divisor
such that $B\equiv C_{0}+2f$. In other words, it suffices to show that $%
\mathcal{E}_{y}\otimes \mathcal{O}_{Y}(xf)$ is generated by global sections
for $x>>0.$ Let us consider $\Bbb{P}(\mathcal{E}_{y}\otimes \mathcal{O}
_{Y}(xf))$ and let us consider its tautological divisor $\mathcal{T},$ we
have only to show that $|\mathcal{T}|$ has no base points.

Let us tensorize $(*)$ by $\mathcal{O}_{Y}(xf).$ We get:

\begin{center}
$0\rightarrow L+xf\rightarrow \mathcal{E}_{y}\otimes \mathcal{O}
_{Y}(xf)\rightarrow (M+xf)\otimes \mathcal{I}_{W}\rightarrow 0.$
\end{center}

Let us fix any fibre $\overline{f}$ of $Y.$ By arguing as in the proof of
Proposition \ref{restriction} $i),$ we have that $[\mathcal{E}_{y}\otimes
\mathcal{\ \ O}_{Y}(xf)]_{|\overline{f}}$ is very ample for $x>>0$. By
arguing as in the proof of Proposition \ref{restriction} $ii)$ and $iii),$
we have that $h^{1}(Y,\mathcal{E}_{y}\otimes \mathcal{O}_{Y}(xf-\overline{f}
))=0$ for $x>>0.$ Hence, we get that $\mathcal{T}_{|\overline{f}}$ is very
ample and $|\mathcal{T}|\rightarrow $\TEXTsymbol{\vert}$\mathcal{T}_{|%
\overline{f}}|$ is surjective. Now, by contradiction, let us assume that $|%
\mathcal{T}|$ has a base point $P$ and let $\overline{f}$ be the unique
fibre of $Y$ passing through $\pi (P);$ $P$ would be a base point also for
\TEXTsymbol{\vert}$\mathcal{T}_{|\overline{f}}|,$ but this is not possible
as $\mathcal{T}_{|\overline{f}}$ is very ample$.$
\end{proof}

\begin{lemma}
\label{lemparete} Let $\mathcal{E}_{y}$ be any vector bundle defined as
above by $(*)$ and let $\overline{f}$ be any fixed fibre of $Y.$ Let $S_{1}$
be the rational ruled surface $\Bbb{P}(\mathcal{E}_{y|\overline{f}})$ and
let $\Gamma _{0}$ and $\varphi $ be the standard generators of \textrm{Num}$%
(S_{1})\simeq \mathrm{Pic}(S_{1}).$ Let $T$ be the tautological divisor of $%
X:=\Bbb{P}(\mathcal{E}_{y})$ as usual and let $A\equiv C_{0}+xf$ be a
divisor on $Y$ with $x>>0.$ We have the following:

$i)$ $S_{1}\simeq \mathbf{F}_{1}$, $T_{|S_{1}}\equiv \Gamma _{0}+2\varphi $
is very ample and for any generic smooth element $S_{2}\in $ $|T+\pi ^{*}A|$
the intersection $S_{2}\cap S_{1}$ is a smooth rational irreducible curve $%
C\equiv \Gamma _{0}+3\varphi ,$ moreover $|T_{|S_{1}}|_{|C}\simeq
|T_{|S_{1}}|;$

$ii)$ let $\xi \in X^{[2]}$ be any subscheme contained in a smooth surface
as $S_{1}$ (i.e. $\xi \in S_{1}^{[2]}$), then $\xi $ is separated by $|T|$
for $-2\leq y\leq 3$ if $h=3$ and for $-2\leq y\leq 4$ if $h=4,$ moreover $%
|T|$ has no base points;

$iii)$ the generic surface $S_{2}\in $ $|T+\pi ^{*}A|$ is isomorphic to the
the blow up of $Y$ at $5x+9+y$ distinct points, hence to the blow up of $%
\Bbb{P}^{2}$ at $5x+10+y$ distinct points; if we generate \textrm{Num}$%
(S_{2})$ with the pull back $l$ of the generator of \textrm{Pic}$(\Bbb{P}%
^{2})$, the pull back $l_{0}$ of $C_{0}$ $\in Y$ and the classes of the
exceptional divisors, we have that $T_{|S_{2}}\equiv
(x+5)l-(x+1)l_{0}-l_{1}\dots -l_{5x+9+y}$, $|T|\simeq |T_{|S_{2}}|$ and $%
h^{0}(S_{2},T_{|S_{2}})$ $\geq $ $11-y.$
\end{lemma}

\begin{proof}
$i)$ The first conclusions follow from Proposition \ref{restriction}, part $%
i).$ Moreover we have that $S_{2|S_{1}}=(T+\pi ^{*}A)_{|S_{1}}\equiv \Gamma
_{0}+2\varphi +\varphi $ and for generic $S_{2}\in $ $|T+\pi ^{*}A|$ the
intersection $C:=$ $S_{2}\cap S_{1}$ is a smooth irreducible curve because $%
T+\pi ^{*}A$ is the tautological divisor of $\Bbb{P}(\mathcal{E}_{y}\otimes
\mathcal{O}_{Y}(A)),$ hence it is a very ample divisor of $X$ by Lemma \ref
{jet}. $C$ is rational being a section of $\mathbf{F}_{1}.$

Let us consider the exact sequence $0\rightarrow T_{|S_{1}}-C\rightarrow
T_{|S_{1}}\rightarrow (T_{|S_{1}})_{|C}\rightarrow 0$ on $S_{1},$ we have $%
h^{i}(S_{1},T_{|S_{1}}-C)=h^{i}(S_{1},\Gamma _{0}+2\varphi -(\Gamma
_{0}+3\varphi ))=h^{i}(S_{1},-\varphi )=0$ for $i=0,1,$ so that $%
|T_{|S_{1}}|_{|C}\simeq |T_{|S_{1}}|,$ note that $h^{0}(S_{1},T_{|S_{1}})=5.$

$ii)$ Let us look at the exact sequence: $0\rightarrow T-S_{1}\rightarrow
T\rightarrow T_{|S_{1}}\rightarrow 0$ on $X.$ We know that $T_{|S_{1}}$ is
very ample in any case, so that $|T_{|S_{1}}|$ separates $\xi ,$ moreover we
have $h^{1}(X,T-S_{1})=0$ for $-2\leq y\leq 3,$ if $h=3,$ by Proposition \ref
{restriction}, part $ii),$ and $h^{1}(X,T-S_{1})=0$ for $-2\leq y\leq 4,$ if
$h=4,$ by Proposition \ref{restriction}, part $iii).$ Hence the map $%
|T|\rightarrow |T_{|S_{1}}|$ is surjective and we can separate $\xi $ by $%
|T|.$

Now, let us assume by contradiction that $|T|$ has a base point $P$ and let $%
f_{P}$ be the fibre of $Y$ passing through $\pi (P).$ As we have seen the
map $|T|\rightarrow |T_{|S_{1}}|$ is surjective, where $S_{1}=\Bbb{P}(%
\mathcal{E}_{y|f_{P}}),$ then $P$ would be a base point for $|T_{|S_{1}}|$
too. But it is not possible because $T_{|S_{1}}$ is very ample.

$iii)$ Recall that, for $x>>0,$ $\mathcal{E}_{y}\otimes \mathcal{O}_{Y}(A)$
is very ample by Lemma \ref{jet}. It is well known that $c_{2}[\mathcal{E}%
_{y}\otimes $ $\mathcal{O}_{Y}(A)]$ is the zero cycle of the zero locus of a
generic section of $\mathcal{E}_{y}\otimes \mathcal{O}_{Y}(A)$ and that a
smooth element $S_{2}\in $ $|T+\pi ^{*}A|$ is isomorphic to the blow up of $%
Y $ exactly at the $\deg \{c_{2}[\mathcal{E}_{y}\otimes $ $\mathcal{O}%
_{Y}(A)]\}=5x+9+y$ points which are the zero locus of the corresponding
section of $\mathcal{E}_{y}\otimes \mathcal{O}_{Y}(A)$ (see \cite{bs}
Theorem 11.1.2.). Moreover $H^{1}(Y,L+A)=0$ for $x>>0,$ hence $H^{0}(Y,%
\mathcal{E}_{y}\otimes \mathcal{O}_{Y}(A))=H^{0}(Y,L+A)\oplus
H^{0}(Y,(M+A)\otimes \mathcal{I}_{W})$ and the zero locus of any section $%
\sigma =\sigma ^{\prime }+\sigma ^{\prime \prime }$ of $\mathcal{E}%
_{y}\otimes \mathcal{O}_{Y}(A)$ is a group of $5x+9+y$ points belonging to $%
(\sigma ^{\prime \prime })_{0}$ (note that this is independent of $h).$\ For
generic $\sigma $ the zero locus $(\sigma ^{\prime \prime })_{0}$ is a
smooth curve and $(\sigma )_{0}$ is a set of points linearly equivalent, on
this curve, to the intersection with $(\sigma ^{\prime })_{0},$ so that we
can assume that they are all distinct.

Let $C_{0}^{\prime },f^{\prime },l_{1},\dots ..,l_{5x+9+y}$ be the
generators of \textrm{Num}$(S_{2})$ (the classes of the pull back of $%
C_{0},f $ and the $5x+9+y$ exceptional divisors of the blow up). The
Wu-Chern relation for $\mathcal{E}_{y}\otimes $ $\mathcal{O}_{Y}(A)$ (see
\cite{gh} pag. 606) implies that $(T+\pi ^{*}A)^{2}$ $=\pi ^{*}\{c_{1}[%
\mathcal{E} _{y}\otimes $ $\mathcal{O}_{Y}(A)]\}(T+\pi ^{*}A)$ $-$ $c_{2}[%
\mathcal{E}_{y}\otimes $ $\mathcal{O}_{Y}(A)]$. Hence $(T+\pi
^{*}A)_{|S_{2}} $ $\equiv (\pi _{|S_{2}})^{*}\{c_{1}[\mathcal{E}_{y}\otimes $
$\mathcal{O}_{Y}(A)]\}$ $- $ $l_{1}\dots .-l_{5x+9+y}$ $\equiv $ $(\pi
_{|S_{2}})^{*}(3C_{0}+5f+2A)-l_{1}\dots .-l_{5x+9+y}$ and $T_{|S_{2}}\equiv $
$(\pi _{|S_{2}})^{*}(3C_{0}+5f+A)-l_{1}\dots .-l_{5x+9+y}\equiv
4C_{0}^{\prime }+(5+x)f^{\prime }-l_{1}\dots .-l_{5x+9+y}.$

As $Y$ is the blow up of $\Bbb{P}^{2}$ at one point $P_{0}$, $S_{2}$ is the
blow up of $\Bbb{P}^{2}$ at $5x+10+y$ points, so we can also generate
\textrm{Num}$(S_{2})$ with the pull back $l$ of the numerical class of a
line in $\Bbb{P}^{2}$ and the classes of the $5x+10+y$ exceptional divisors.
If $l_{0}$ is the class of the pull back of the exceptional divisor of the
blow up of $\Bbb{P}^{2}$ at $P_{0}$ we have $C_{0}^{\prime }\equiv l_{0}$
and $f^{\prime }\equiv $ $l-l_{0}$ so that: $T_{|S_{2}}\equiv
(x+5)l-(x+1)l_{0}-l_{1}\dots .-l_{5x+9+y}.$ It is easy to see that $%
h^{0}(S_{2},T_{|S_{2}})$ $\geq $ $11-y.$ Note that $h^{0}(S_{2},T_{|S_{2}})$
$=$ $11-y$ if the $5x+10+y$ points are in general position, but this fact is
not known a priori.

Now let us consider the exact sequence: $0\rightarrow T-S_{2}\rightarrow
T\rightarrow T_{|S_{2}}\rightarrow 0\;$on $X.$ As $T-S_{2}=-\pi ^{*}A$ we
have that $h^{0}(X,-\pi ^{*}A)=0$ and $h^{1}(X,-\pi ^{*}A)=h^{1}(Y,-A)=0,$
so that $H^{0}(X,T)=H^{0}(S_{2},T_{|S_{2}})$ hence $|T|\simeq |T_{|S_{2}}|.$
\end{proof}

\begin{lemma}
\label{lemtg} Let $\mathcal{E}_{y}$ be any vector bundle defined as above by
$(*)$. Let $\xi \in X^{[2]}$ be any subscheme of $X$ having its support at a
single point $P\in X.$ Then $|T|$ separates $\xi $ for $-2\leq y\leq 3$ if $%
h=3$ and for $-2\leq y\leq 4$ if $h=4.$
\end{lemma}

\begin{proof}
We will show that assumptions $\beta $) and $\gamma $) of Proposition \ref
{tre} are satisfied with $A\equiv C_{0}+xf$, $x>>0,$ and $B\equiv f.$ It is
easy to see that $\gamma $) it is true for $x>>0,$ so we have to prove that $%
\beta $) is true for any point $P$ and any direction \underline{$q$}$\in
\Bbb{T}_{P}(X).$ In other words, we have to show that there exists a smooth
section $\tau \in $ $|T|$ such that $(\tau )_{0}$ passes through $P$ and its
tangent space at $P$ does not contain $\underline{q}$.

Let us consider $\pi (P)$ and the fibre $f_{P}$ of $Y$ passing through $\pi
(P)$. Let us choose a smooth surface $S_{2}\in $ $|T+\pi ^{*}A|$ passing
through $P$ (recall that $T+\pi ^{*}A$ is very ample for $x>>0$) and a
smooth $S_{1}:=\pi ^{-1}(f_{P});$ so $\beta $)$i$) is fullfilled.

By Lemma \ref{lemparete} $ii)$ we can assume that $\underline{q}\notin \Bbb{T%
}_{P}(S_{1})$.

By recalling that $T_{|S_{1}}$is very ample in any case we can choose $%
\sigma _{1}\in $ $|T_{|S_{1}}|$ such that $(\sigma _{1})_{0}$ is a smooth
curve, passing through $P,$ with tangent vector $\underline{t}\in $ $\Bbb{T}%
_{P}(S_{1})$ and $\beta $)$ii$) is fulfilled. Obviously $\underline{t}\neq $
$\underline{q}.$ We can also choose a generic smooth element $S_{2}\in
|T+\pi ^{*}A|$ such that $S_{2}$ cuts transversely $S_{1}$ along a smooth
rational curve $C$, (hence $\Bbb{T}_{P}(S_{1})\cap \Bbb{T}_{P}(S_{2})=\Bbb{T}%
_{P}(C)$) and $S_{2}$ cuts transversely $(\sigma _{1})_{0}$ at $4$ distinct
points $P=R_{1},\dots ,R_{4}$ $\in C$ by Lemma \ref{lemparete} $i).$ The
independent vectors $\underline{t}$ and $\underline{q}\in \Bbb{T}_{P}(X)$
generates a $2$-plane in $\Bbb{T}_{P}(X)$ cutting $\Bbb{T}_{P}(S_{2})$ along
a vector \underline{$w$}. Now we choose a vector \underline{$v$}$\in \Bbb{T}%
_{P}(S_{2}),$ \underline{$v$}$\neq $\underline{$w$}, and we consider the
linear subsystem $\Lambda $ of $|T_{|S_{2}}|=$ $|(x+5)L-(x+1)l_{0}-l_{1}%
\dots .-l_{5x+9+y}|$ (see Lemma \ref{lemparete} $iii)$) given by those
elements passing through $P=R_{1},\dots ,R_{4}$ and such that their
zero-loci are tangent to \underline{$v$} at $P$. This is possible because $
h^{0}(X,T)=h^{0}(S_{2},T_{|S_{2}})\geq 11-y,$ (see Lemma \ref{lemparete} $%
iii)$), and therefore the (projective) dimension of $\Lambda $ is $\geq
5-y\geq 1$.

Now we want to show that not all the zero loci of the elements of $\Lambda $
are singular at $P.\;$In fact any element of $\Lambda $ whose zero-locus is
singular at $P$ comes from a degree $x+5$ plane curve $\overline{C}$
intersecting a line $\overline{l}$ passing through $P_{0}$ (corresponding to
$C$) with multiplicity at least $x+6$.\ Hence $\overline{l}$ is an
irreducuble component of $\overline{C}$ and the zero-locus of the
corresponding element of $\Lambda $ contains $C.$ Therefore the zero-loci of
all elements of $\Lambda $ are singular at $P$ if and only if $\Lambda
=|T_{|S_{2}}-C|.$ But this is a fixed subspace of $|T_{|S_{2}}|,$ whose
sections have zero-loci having at most a finite number of fixed tangent
vectors at $P.$ If $\Lambda =|T_{|S_{2}}-C|$ all zero-loci of $\Lambda $ are
reducible as the union of $C$ and other curves passing through $P$ with the
same tangent vector \underline{$v$}$,$ so that it suffices to change
suitably the choice of \underline{$v$} to avoid this case.

In conclusion we can assume that not all zero-loci of the elements of $%
\Lambda $ are singular at $P,$ hence that not all of them contain $C.$ Hence
there exists a section $\sigma _{2}\in \Lambda \subseteq |T_{|S_{2}}|$ whose
zero-locus is smooth at $P,$ having \underline{$v$} as tangent vector, such
that \underline{$w$}$\notin <\underline{t},\underline{v}>,$ so that
\underline{$q$}$\notin <\underline{t},\underline{v}>$ too so that $\beta $)$%
iii$) is fullfilled.

As $H^{0}(X,T)=H^{0}(S_{2},T_{|S_{2}})$ we get that there exists $\tau \in
|T|$ such that $\tau _{|S_{2}}=$ $\sigma _{2}$ and $(\tau )_{0}$ does not
contain $S_{1}$ because $(\sigma _{2})_{0}$ does not contain $C.\;$ Hence $%
\sigma _{1}^{^{\prime }}:=\tau _{|S_{1}}$ is a non zero section of $%
|T_{|S_{1}}|$ and we have $\sigma _{1}^{^{\prime }}\in $ $<\sigma _{1}$ $>$
because $\sigma _{1}^{^{\prime }}$ and $\sigma _{1}$ cut the same divisor $%
R_{1}+R_{2}+R_{3}+R_{4}$ on $C$ and $|T_{|S_{1}}|_{|C}\simeq |T_{|S_{1}}|$
by Lemma \ref{lemparete} $i).$ By choosing suitably the generator of $%
<\sigma _{1}$ $>$ we can assume that $\tau _{|S_{1}}=\sigma _{1}$ so that $%
\beta $)$iv$) is fulfilled a fortiori by choosing $\sigma :=\tau
_{|S_{1}\cup S_{2}}$.
\end{proof}

Now we can prove the following result, stating the very ampleness of almost
all vector bundles $\mathcal{E}_{y}$ when $h=3.$

\begin{theorem}
\label{teoy3} Let $\mathcal{E}_{y}$ be any vector bundle defined as above by
$(*)$ with $-2\leq y\leq 2$ and $h=3.$ Then $\mathcal{E}_{y}$ is very ample.
\end{theorem}

\begin{proof}
Let $T$ be the tautological divisor of $X=\Bbb{P}(\mathcal{E}_{y})$, let $%
\xi $ be any fixed element of $X^{[2]}$ and let $s:Y\rightarrow \Bbb{P}^{2}$
be the blow up of $\Bbb{P}^{2}$ at $P_{0}.$ We have to prove that $|T|$
separates $\xi .$ By Lemmas \ref{lemtg} and \ref{lemparete} we know that we
have to consider only the cases in which the support of $\xi $ consists of a
couple of distinct points $P,Q$ projecting on different fibres of $Y.$

\textbf{Case 1:} neither $\pi (P)$ nor $\pi (Q)$ belong to $C_{0}.$ We use a
slightly different version of the proof of Theorem \ref{Valma} i.e. we use
the linear system $|C_{0}+f|$ which is not very ample. However, in this
case, there exists a smooth element $\gamma \in |C_{0}+f|$ passing through $%
\pi (P)$ and $\pi (Q).$ Moreover $\gamma $ passes through two points of $W$
at most, because $\gamma $ corresponds to the unique line passing through $%
s[\pi (P)]$ and $s[\pi (Q)]$ on $\Bbb{P}^{2}$ and $s(W)$ is a set of points
in general position on $\Bbb{P}^{2}.$ Very ampleness of $\mathcal{E}_{y}$
follows from Proposition \ref{restriction} $vii)$ and $viii)$ as $T_{|\Bbb{P}%
(\mathcal{E}_{y|\gamma })}$ is very ample and $|T|\rightarrow |T_{|\Bbb{P}(%
\mathcal{E}_{y|\gamma })}|$ is surjective.

\textbf{Case 2:} $\pi (P)$ and $\pi (Q)$ belong to $C_{0}.$ Very ampleness
of $\mathcal{E}_{y}$ follows from Proposition \ref{restriction} $iv)$ and $%
vi)$ as $T_{|\Bbb{P}(\mathcal{\ \ E}_{y|C_{0}})}$ is very ample and $%
|T|\rightarrow |T_{|\Bbb{P}(\mathcal{E}_{y|C_{0}})}|$ is surjective.

\textbf{Case 3:} $\pi (Q)$ $\in C_{0}$ and $\pi (P)\notin C_{0}$. Let $f_{P}$
be the fibre of $Y$ passing through $\pi (P)$ and let $\Sigma \subset X$ be
the reducible surface $\Bbb{P}(\mathcal{E}_{y|C_{0}})\cup \Bbb{P}(\mathcal{E}%
_{y|f_{P}}).$ Let $R\in Y$ be the unique point $f_{P}\cap C_{0}$ so that $%
\Bbb{P}(\mathcal{E}_{y|C_{0}})\cap \Bbb{P}(\mathcal{E}_{y|f_{P}})=\mathrm{F}%
_{R}$ (as in the proof of Proposition \ref{tre}). By Proposition \ref
{restriction} $iv),$ we know that $T_{|\Bbb{P}(\mathcal{E}_{y|C_{0}})}$ is
very ample, so that we can take a smooth element $\sigma _{2}\in |T_{|\Bbb{P}%
(\mathcal{E}_{y|C_{0}})}|$ such that $Q\notin (\sigma _{2})_{0}$ and $%
(\sigma _{2})_{0}$ cuts $\mathrm{F}_{R}$ transversely at a point $H.$ By
Proposition \ref{restriction} $i),$ and $ii),$ we know that $T_{|\Bbb{P}(%
\mathcal{E}_{y|f_{P}})}$ is very ample and that $|T_{|\Bbb{P}(\mathcal{E}
_{y|f_{P}})}|=|\Gamma _{0}+2\varphi |$ on a surface $\mathbf{F}_{1}.$ It is
possible to choose an element (not necessarily smooth) $\sigma _{1}\in |T_{|%
\Bbb{P}(\mathcal{E}_{y|f_{P}})}|$ such that $(\sigma _{1})_{0}$ passes
through $P$ and cuts $\mathrm{F}_{R}$ transversely at $H.$ Note that $%
\mathbf{F}_{1}$ is embedded by $|\Gamma _{0}+2\varphi |$ as a scroll in $%
\Bbb{P}^{4}$, in such a way that $\mathrm{F}_{R}$ is a fibre of the scroll,
but $P\notin \mathrm{F}_{R}$, hence it is not possible that all hyperplanes
passing through $P$ and $H$ contain the line $\mathrm{F}_{R}$ in $\Bbb{P}%
^{4}.$

Now the pair $(\sigma _{1},\sigma _{2})$ is a section of $|T_{|\Sigma }|$
separating $P$ from $Q$ and, by Proposition \ref{restriction} $viii),$ we
can lift this element to an element of $|T|$ acting in the same way and we
are done. In fact, note that the proof of Proposition \ref{restriction} $%
viii)$ works even when $\gamma =C_{0}\cup f_{P}.$

Obviously if $\pi (Q)$ $\notin C_{0}$ and $\pi (P)\in C_{0}$ we can
interchange the roles of $P$ and $Q$ in the previous argument.
\end{proof}

We can also prove the very ampleness of $\mathcal{E}_{3}$ when $h=4,$ but we
need other Lemmas.

\begin{lemma}
\label{lemReider} Let $P_{0},P_{1},...,P_{8}$ be $9$ distinct points in $%
\Bbb{P}^{2}$, lying on a smooth cubic curve $\mathcal{C}.$ Assume that the
complete linear system $\mathcal{L}$ of quartics passing through them has no
base points$.$ Then $\mathcal{L}$ is very ample.
\end{lemma}

\begin{proof}
Let $S$ be the blow up of $\Bbb{P}^{2}$ at $P_{0},P_{1},...,P_{8}.$ We can
generate \textrm{Pic}$(S)\simeq \mathrm{Num}(S)$ with the pull back $l$ of
the numerical class of a line in $\Bbb{P}^{2}$ and the classes $l_{i}$ of
the $9$ exceptional divisors. In this notation the class of any irreducible
curve $\gamma $ on $S$ is either one of the $l_{i}$ or a class of the
following type: $al-a_{0}l_{0}-...-a_{8}l_{8},$ for suitable integers $a\geq
1$ and $a_{i}\geq 0,$ as $\gamma $ comes from an irreducible plane curve.
The curve $\mathcal{C}$ in $\Bbb{P}^{2}$ gives rise to a curve $\overline{%
\mathcal{C}}$ on $S$ such that $\overline{\mathcal{C}}\equiv
3l-l_{0}-...-l_{8},$ moreover for any irreducible curve $\gamma $ on $S,$
different from $l_{i},$ we have: $0\leq \overline{\mathcal{C}}\gamma
=3a-a_{0}-...-a_{8};$ it follows that the class $%
al-a_{0}l_{0}-...-a_{8}l_{8} $ of any irreducible curve $\gamma $ on $S,$
different from $l_{i},$ must satisfy the condition: $3a\geq a_{0}+...+a_{8}.$

Very ampleness of $\mathcal{L}$ is equivalent to very ampleness of $%
|4l-l_{0}-l_{1}....-l_{8}|$ on $S$ and this will be established via Reider's
method (see \cite{dl} Theorem 2.1). Let $M$ be a divisor on $S$ such that $%
4l-l_{0}-l_{1}....-l_{8}\equiv K_{S}+M.$ It is $M\equiv
7l-2l_{0}-2l_{1}....-2l_{8}$. To be able to apply Reider's Theorem, $M$ must
be big and nef with $M^{2}\geq 10.$ Obviously $M^{2}=13,$ moreover $%
Ml_{i}=2, $ for any $i,$ and $M\gamma \geq a\geq 1$ for any other
irreducible curve $\gamma $ on $S,$ thanks to the above condition. It
follows that $M$ is ample (Nakai-Moishezon criterion, see \cite{h} pag. 365)
and therefore big and nef.

Now, if $E$ is a candidate effective divisor that, according to Reider,
could cause $K_{S}+M$ not to be very ample, it must be as in one of these
cases:

$1)$ $E\equiv l-\sum\limits_{i=0}^{8}a_{i}l_{i},$ $E$ irreducible, $ME=1,$ $%
\sum\limits_{i=0}^{8}a_{i}=3,$ $0\leq a_{i}\leq 1$ for any $i;$

$2)$ $E\equiv 2l-\sum\limits_{i=0}^{8}a_{i}l_{i},$ $E$ irreducible, $ME=2,$ $%
\sum\limits_{i=0}^{8}a_{i}=6;$

$3)$ $E=E_{1}+E_{2}$, $E_{1}\neq E_{2},$ where each $E_{j}$ is irreducible, $%
E_{j}\equiv l-\sum\limits_{i=0}^{8}a_{ij}l_{i},$ $ME=2,$ $%
\sum\limits_{i=0}^{8}a_{ij}=3,$ $0\leq a_{ij}\leq 1$ for any $i$ and $j.$

In all cases it is $E^{2}\leq -2,$ not satisfying Reider's conditons, hence $%
\mathcal{L}$ is very ample.
\end{proof}

\begin{lemma}
\label{claim} Let $\mathcal{E}_{y}$ be any vector bundle defined as above by
$(*)$ with $y=3$ and $h=4$ (hence $w=7$). Let $\overline{f}$ be any fixed
fibre of $Y$ and let $S_{1}:=\Bbb{P}(\mathcal{E}_{3|\overline{f}})\subset X.$
Then:

$i)$ the generic element $S$ of the linear system $|T-S_{1}|$ is smooth and
irreducible;

$ii)$ any linear subsystem of $|T-S_{1}|$ consisting of elements which are
all singular and or reducible has codimension at least two;

$iii)$ the restriction map $|T|\rightarrow |T|_{|S}$ is surjective for any
smooth $S\in $ $|T-S_{1}|$;

$iv)$ $T_{|S}$ is very ample for any generic smooth $S\in $ $|T-S_{1}|$.
\end{lemma}

\begin{proof}
For simplicity let us write $\mathcal{E}$ instead of $\mathcal{E}_{3}.$ By $%
(*)$ it follows that $h^{0}(Y,\mathcal{E})=8.$ By tensorizing $(*)$ with $%
\mathcal{O}_{Y}(-$ $\overline{f})$ we get that $h^{0}(Y,\mathcal{E}\otimes
\mathcal{O}_{Y}(-\overline{f}))=h^{0}(X,T-S_{1})=3.$

$i)$ and $ii).$ Let $\Lambda $ be any linear subsystem of $|T-S_{1}|$ such
that every element of $\Lambda $ is singular or reducible. To prove $i)$ and
$ii)$ we have to show that $\dim (\Lambda )\leq 0.$ Recalling the proof of
Proposition \ref{due}, we know that any element of $|T-S_{1}|$ is singular
if and only if it is reducible and that $\Lambda \neq \emptyset $ if and
only if there exists an effective divisor $D=aC_{0}+\rho ^{*}B\in \mathrm{Pic%
}(Y),$ with $\deg (B)=b,$ $a\geq 0,$ $b\geq 0,$ such that $T-S_{1}-\pi ^{*}D$
is effective and, in this case, $\Lambda =|T-S_{1}-\pi ^{*}D|+|\pi ^{*}D|$
with $h^{0}(X,T-S_{1}-\pi ^{*}D)=1$ and $h^{0}(X,\pi ^{*}D)-1=\dim (\Lambda
) $ or $h^{0}(X,\pi ^{*}D)=1$ and $h^{0}(X,T-S_{1}-\pi ^{*}D)-1=\dim
(\Lambda ) $.

From the exact sequence:\newline
$0\rightarrow C_{0}+f-\overline{f}-D\rightarrow \mathcal{E}\otimes \mathcal{O%
}_{Y}(-\overline{f}-D)\rightarrow (2C_{0}+4f-\overline{f}-D)\otimes \mathcal{%
I}_{W}\rightarrow 0$\newline
we see that $h^{0}(X,T-S_{1}-\pi ^{*}D)=h^{0}(Y,\mathcal{E}\otimes \mathcal{O%
}_{Y}(-\overline{f}-D)$ can be positive only if $a\leq 1.$

Let us assume $a=1.$ In this case we have to consider the exact sequence:%
\newline
$0\rightarrow -bf\rightarrow \mathcal{E}\otimes \mathcal{O}_{Y}(-\overline{f}
-D)\rightarrow (C_{0}+3f-bf)\otimes \mathcal{I}_{W}\rightarrow 0.$

As $P_{0}\cup s(W)$ are in general position it is easy to see that $%
h^{0}(Y,(C_{0}+3f-bf)\otimes \mathcal{I}_{W})=0$ for any $b\geq 0$ . If $b>0$
$h^{0}(X,T-S_{1}-\pi ^{*}D)=h^{0}(Y,\mathcal{E}\otimes \mathcal{\ O}_{Y}(-%
\overline{f}-D)=0;$ if $b=0$ $h^{0}(X,T-S_{1}-\pi ^{*}D)=h^{0}(Y,\mathcal{E}%
\otimes \mathcal{O}_{Y}(-\overline{f}-D)=1$ and $h^{0}(X,\pi
^{*}D)=h^{0}(Y,D)=1,$ so that $\dim (\Lambda )=0$ and we are done.

Let us assume $a=0$ (hence $b\geq 1).$ In this case we have to consider the
exact sequence: \newline
$0\rightarrow C_{0}-bf\rightarrow \mathcal{E}\otimes \mathcal{O}_{Y}(-%
\overline{f}-D)\rightarrow (2C_{0}+3f-bf)\otimes \mathcal{I}_{W}\rightarrow
0.$

As $P_{0}\cup s(W)$ are in general position it is easy to see that $%
h^{0}(Y,(2C_{0}+3f-bf)\otimes \mathcal{I}_{W})=0,$ moreover $%
h^{0}(Y,C_{0}-bf)=0,$ so that $h^{0}(X,T-S_{1}-\pi ^{*}D)=h^{0}(Y,\mathcal{E}
\otimes \mathcal{O}_{Y}(-\overline{f}-D)=0$ and $\Lambda =\emptyset .$

$iii)$ Let $S$ be a smooth element of $|T-S_{1}|$ and let us consider the
exact sequence: $0\rightarrow T-S\rightarrow T\rightarrow T_{|S}\rightarrow
0.\;$ As $T-S=\pi ^{*}\overline{f}$ we have that $h^{0}(X,\pi ^{*}\overline{f%
})=h^{0}(Y,\overline{f})=2$ and $h^{1}(X,\pi ^{*}\overline{f})=h^{1}(Y,%
\overline{f})=0,$ so that $h^{0}(X,T-S)=2,$ $h^{0}(S,T_{|S})=8-2=6$ and the
map $H^{0}(X,T)\rightarrow H^{0}(S,T_{|S})$ is surjective.\

$iv)$ $S,$ being a smooth generic element of $|T-\pi ^{*}\overline{f}|,$ is
isomorphic to the blow up of $Y$ at $\deg \{c_{2}[\mathcal{E}\otimes $ $%
\mathcal{O}_{Y}(-\overline{f})]\}=8$ distinct points (see the proof of Lemma
\ref{lemparete}, $iii)$). Let $C_{0}^{\prime },f^{\prime },l_{1},.....,l_{8}$
be the generators of \textrm{Num}$(S)$ as in Lemma \ref{lemparete} $iii)$
for $S_{2}$. The Wu-Chern relation for $\mathcal{E}\otimes $ $\mathcal{O}%
_{Y}(-\overline{f})$ implies that $(T-\pi ^{*}\overline{f})^{2}$ $=\pi
^{*}\{c_{1}[\mathcal{E}\otimes $ $\mathcal{O}_{Y}(-\overline{f})]\}(T-\pi
^{*}\overline{f})$ $-$ $c_{2}[\mathcal{E}\otimes $ $\mathcal{O}_{Y}(-%
\overline{f})]$. Hence $(T-\pi ^{*}\overline{f})_{|S}$ $\equiv (\pi
_{|S})^{*}\{c_{1}[\mathcal{E}\otimes $ $\mathcal{O}_{Y}(-\overline{f})]\}$ $%
- $ $l_{1}....-l_{8}$ $\equiv $ $(\pi
_{|S})^{*}(3C_{0}+5f-2f)-l_{1}....-l_{8}$ and $T_{|S}\equiv $ $(\pi
_{|S})^{*}(3C_{0}+5f-f)-l_{1}....-l_{8}\equiv 3C_{0}^{\prime }+4f^{\prime
}-l_{1}....-l_{8}.$

As $Y$ is the blow up of $\Bbb{P}^{2}$ at one point $P_{0}$, $S$ is
isomorphic to the blow up of $\Bbb{P}^{2}$ at $9$ distinct points, so we can
also generate \textrm{Num}$(S)$ with the pull back $l$ of the numerical
class of a line in $\Bbb{P}^{2}$ and the classes of the $9$ exceptional
divisors. If $l_{0}$ is the class of the pull back of the exceptional
divisor of the blow up of $\Bbb{P}^{2}$ at $P_{0}$ we have $C_{0}^{\prime
}\equiv l_{0}$ and $f^{\prime }\equiv $ $l-l_{0}$ so that: $T_{|S}\equiv
4l-l_{0}-l_{1}....-l_{8}.$

To show that $T_{|S}$ is very ample we can apply Lemma \ref{lemReider}: $%
h^{0}(S,T_{|S})=6$ by $iii);$ $|T_{|S}|$ does not have base points, because $%
|T|$ does not have base points by Lemma \ref{lemparete} $ii)$ and $
|T|\rightarrow |T_{|S}|$ is surjective; the $9$ distinct points lie on a
smooth cubic because the zero-locus of a generic section of $\mathcal{E}%
\otimes $ $\mathcal{O}_{Y}(-\overline{f})$ is a set of $8$ distinct points
belonging to a generic element of the linear system $|(2C_{0}+3f)\otimes
\mathcal{I}_{W}|$ $=|(3l-l_{0})\otimes \mathcal{I}_{W}|$ on $Y,$ (see the
proof of Lemma \ref{lemparete}, $iii)$) corresponding to a smooth plane
cubic curve passing through $P_{0}\cup s(W),$ where $s:Y\rightarrow \Bbb{P}%
^{2}$ is the blow up.
\end{proof}

\begin{theorem}
\label{teoy4} Let $\mathcal{E}_{y}$ be any vector bundle defined as above by
$(*)$ with $y=3$ and $h=4$ (hence $w=7$), then $\mathcal{E}_{3}$ is very
ample.
\end{theorem}

\begin{proof}
As before, let us write $\mathcal{E}$ instead of $\mathcal{E}_{3}.$ Let $T$
be the tautological divisor of $X=\Bbb{P}(\mathcal{E})$ and let $\xi $ be
any fixed element of $X^{[2]}.$ We have to prove that $|T|$ separates $\xi .$
By Lemmas \ref{lemtg} and \ref{lemparete} we know that we have to consider
only the cases when $\xi $ is a couple of distinct points $P,Q$ projecting
on different fibres of $Y.$

We want to apply Proposition \ref{tre}, so that we will prove that
assumptions $\alpha $) and $\gamma $) of Proposition \ref{tre} are satisfied
with $A\equiv C_{0}+xf$, $x>>0,$ and $B\equiv f.$ It is easy to see that $%
\gamma $) is true for $x>>0,$ so we have to prove that $\alpha $) is true
for any couple of distinct points $P,Q\in X.$ In other words, we have to
show that there exists a section $\tau \in $ $|T|$ such that $(\tau )_{0}$
passes through $P$ and does not pass through $Q.$ Of course we can change
the role of $P$ and $Q$ to separate $Q$ from $P.$

Let us consider $\pi (Q)$ and the fibre $f_{Q}$ of $Y$ passing through $\pi
(Q)$. Let us choose a smooth surface $S_{2}\in $ $|T+\pi ^{*}A|$ passing
through $P$ and not through $Q$ (recall that $T+\pi ^{*}A$ is very ample for
$x>>0$) and a smooth $S_{1}:=\pi ^{-1}(f_{Q});$ so $\alpha $)$i)$ is
fullfilled, moreover we can assume that $S_{2}$ cuts $S_{1}$ transversely
along a smooth rational curve $C$ as in the proof of Lemma \ref{lemtg}.

As $T_{|S_{1}}$ is very ample by Lemma \ref{lemparete} $i),$ we can choose a
section $\sigma _{1}\in |T_{|S_{1}}|$ such that $(\sigma _{1})_{0}$ is
smooth, does not pass through $Q$ and cuts $C$ at $4$ distinct points $%
R_{1},...,R_{4}$ (see the proof of Lemma \ref{lemtg}) so that $\alpha $)$ii)$
is fulfilled.

To get $\alpha $)$iii)$ and $iv)$ we look for a section $\sigma _{2}\in
H^{0}(S_{2},T_{|S_{2}})$ whose zero locus passes through $P,R_{1},...,R_{4}$
and does not contain $C$. By Lemma \ref{lemparete} $iii)$ we know that $%
S_{2} $ is isomorphic to the the blow up of $Y$ at $5x+12$ distinct points,
hence to the blow up of $\Bbb{P}^{2}$ at $5x+13$ distinct points; if we
generate \textrm{Num}$(S_{2})$ with the pull back $l$ of the generator of
\textrm{Pic}$(\Bbb{P}^{2})$, the pull back $l_{0}$ of $C_{0}$ $\in Y$ and
the classes of the exceptional divisors, we have that $T_{|S_{2}}\equiv
(x+5)l-(x+1)l_{0}-l_{1}\dots -l_{5x+12}$, $|T|\simeq |T_{|S_{2}}|$ and $%
h^{0}(S_{2},T_{|S_{2}})$ $=h^{0}(X,T)=h^{0}(Y,\mathcal{E})=8.$

Let us consider the linear subspace $H^{0}(S_{2},T_{|S_{2}}\otimes \mathcal{I%
}_{P})$ of $H^{0}(S_{2},T_{|S_{2}})$ given by sections whose zero locus
contains $P;$ $\dim [H^{0}(S_{2},T_{|S_{2}}\otimes \mathcal{I}_{P})]=7,$
(recall that $|T|\simeq |T|_{|S_{2}}=|T_{|S_{2}}|$ has no base points by
Lemma \ref{lemparete} $ii)$ ). Let us consider the restriction $\rho _{P}$
of the natural map $\rho :H^{0}(S_{2},T_{|S_{2}})\rightarrow
H^{0}(C,T_{|C})\simeq H^{0}(\Bbb{P}^{1},\mathcal{O}_{\Bbb{P}^{1}}(4))\simeq
\Bbb{C}^{5}\;$to $H^{0}(S_{2},T_{|S_{2}}\otimes \mathcal{I}_{P}).$

We claim that $\rho $ is surjective. Indeed, consider first the structure
sequence of $S_{2}$ on $X,$ tensored with $T-S_{1}$:\newline
$0\rightarrow T-S_{1}-S_{2}\rightarrow T-S_{1}\rightarrow
(T-S_{1})_{|S_{2}}\rightarrow 0.$ It is: $h^{1}(X,T-S_{1})=0$ by Proposition
\ref{restriction} $iii)$ and $h^{2}(X,T-S_{1}-S_{2})=h^{2}(X,-\pi
^{*}(A+f_{Q}))=h^{2}(Y,-A-f_{Q})=h^{0}(Y,K_{Y}+A+f_{Q})=0,$ hence $%
h^{1}(S_{2},(T-S_{1})_{|S_{2}})=0.$ \newline
Let us then consider the structure sequence of $C$ on $S_{2}$ tensored with $%
T_{|S_{2}}.$ As $h^{1}(S_{2},(T-S_{1})_{|S_{2}})=h^{1}(S_{2},T_{|S_{2}}-C)=0$
it follows that $\rho $ is surjective, so that our claim is proved, moreover
$\ker (\rho )=H^{0}(S_{2},T_{|S_{2}}-C).$ As $%
3=h^{0}(X,T-S_{1})=h^{0}(S_{2},(T-S_{1})_{|S_{2}})=$ $%
h^{0}(S_{2},T_{|S_{2}}-C),$ it follows that $\dim [\ker (\rho )]=3.$

Now let us consider the following two cases.

\textbf{Case 1:} let us assume that $P$ is in the base locus of $%
|T_{|S_{2}}-C|=|(T-S_{1})_{|S_{2}}|=|T-S_{1}|_{|S_{2}},$ hence in the base
locus of $|T-S_{1}|$ because $|T|\simeq |T|_{|S_{2}}=|T_{|S_{2}}|.$ If $Q$
is not in the base locus of $|T-S_{1}|,$ then there exists an element $%
\widetilde{S}\in |T-S_{1}|=|T-\pi ^{*}f|$ passing through $P$ and not
passing through $Q.$ Let us pick $S_{1}^{\prime }:=\pi ^{-1}(\overline{f}),$
where $\overline{f}$ is a fibre of $Y$ different from $f_{Q}$ and we get an
element $\widetilde{S}\cup $ $S_{1}^{\prime }\in $ $|T-\pi ^{*}f|+|\pi
^{*}f|\subseteq |T|$ separating $P$ from $Q$ without using Proposition \ref
{tre}. If $Q$ is in the base locus of $|T-S_{1}|,$ let us pick a generic
smooth surface $S\in $ $|T-S_{1}|,$ obviously passing through $P$ and $Q,$
and existing by Lemma \ref{claim} $i)$ and $ii)$. Now, by Lemma \ref{claim} $%
iii)$ and $iv)$ we can separate $P$ from $Q$ by $|T|$ directly, without
using Proposition \ref{tre}.

\textbf{Case 2:} let us assume that $P$ is not in the base locus of the
linear system $|T_{|S_{2}}-C|=|(T-S_{1})_{|S_{2}}|=|T-S_{1}|_{|S_{2}}.$ Then
there exists at least an element of $|T_{|S_{2}}-C|$ whose zero-locus does
not pass through $P,$ hence $H^{0}(S_{2},T_{|S_{2}})_{P}\nsupseteqq
H^{0}(S_{2},T_{|S_{2}}-C)=\ker (\rho )$ and $\dim [\ker (\rho _{P})]=\dim
[H^{0}(S_{2},T_{|S_{2}})_{P}\cap H^{0}(S_{2},T_{|S_{2}}-C)]=2,\;$(a priori $%
\dim [\ker (\rho _{P})]=\{2,3\}$). Therefore $\rho _{P}$ is surjective too
and we can choose a section $\sigma _{2}\in $ $|T_{|S_{2}}|$ whose zero
locus passes through $P,$ $R_{1},...,R_{4}$, not containing $C.$ In this
case we can conclude as in the proof of Lemma \ref{lemtg}: as $%
H^{0}(X,T)=H^{0}(S_{2},T_{|S_{2}})$ we get that there exists $\tau \in |T|$
such that $\tau _{|S_{2}}=$ $\sigma _{2}$ and $(\tau )_{0}$ does not contain
$S_{1}$ because $(\sigma _{2})_{0}$ does not contain $C.\;$ Hence $\sigma
_{1}^{^{\prime }}:=\tau _{|S_{1}}$ is a non zero section of $|T_{|S_{1}}|$
and we have $\sigma _{1}^{^{\prime }}\in $ $<\sigma _{1}$ $>$ because $%
\sigma _{1}^{^{\prime }}$ and $\sigma _{1}$ cut the same divisor $%
R_{1}+R_{2}+R_{3}+R_{4}$ on $C$ and $|T_{|S_{1}}|_{|C}\simeq |T_{|S_{1}}|$
by Lemma \ref{lemparete} $i).$ By choosing suitably the generator of $%
<\sigma _{1}$ $>$ we can assume that $\tau _{|S_{1}}=\sigma _{1}$ so that $%
\alpha $)$iii$) and $\alpha $)$iv$) are fulfilled a fortiori by choosing $%
\sigma :=\tau _{|S_{!}\cup S_{2}}.$
\end{proof}

\section{Existence and non existence of some $3$-folds}

The study of linearly normal projective manifolds of low degree got a boost
as a result of classical adjunction theory, as developed by Sommese and his
collaborators. The approach consists of three phases: enumeration of all
possible manifolds of given degree according to their adjunction theoretic
structure and values of numerical characters; investigation of actual
effective existence of elements appearing in the compiled lists; study of
the Hilbert scheme of existing manifolds (see \cite{bb2} for details). In
\cite{fl1}, \cite{fl2}, \cite{bb1}, such a study is conducted for degree,
respectively, $9,10,11.$ In all three papers the existence of members of a
particular family of $3$-fold scrolls was left as an open problem. They are
scrolls of the form $X:=\Bbb{P}(\mathcal{E}_{y}),$ of degree $[c_{1}(%
\mathcal{E}_{y})]^{2}-c_{2}(\mathcal{E}_{y})=13-y,$ where $\mathcal{E}_{y}$
is a rank $2$ vector bundle over $Y=\mathbf{F}_{1},$ having $c_{1}(\mathcal{E%
}_{y})\equiv 3C_{0}+5f$ and $c_{2}(\mathcal{E}_{y})=8+y,$ with $y=2,3,4.$

The analysis conducted in Section 6 gives immediately the following::

\begin{corollary}
\label{finalsi} There exist linearly normal $3$-folds $\Bbb{P}(\mathcal{E}%
_{y})$, $y=3,2$, where $\mathcal{E}_{y}$ is a rank two vector bundle given
by $(*),$ embedded as linear scrolls over $\mathbf{F}_{1}$, with $c_{1}(%
\mathcal{E}_{y})\equiv 3C_{0}+5f$, $c_{2}(\mathcal{E}_{y})=8+y$ and degree,
respectively, $10$ and $11$.
\end{corollary}

\begin{proof}
Apply Theorem \ref{teoy3} and Theorem \ref{teoy4}, by recalling that vector
bundles $\mathcal{E}_{y}$ defined by $(*)$ have the prescribed Chern classes.
\end{proof}

In fact Theorem. \ref{teoy3} proves the existence of other $3$-folds of the
same type. It is natural to ask if rank $2$ very ample vector bundles over $%
Y $ can be defined by using $(*)$ for other values of $y.$ The answer is
negative for $y=4,$ $h=4.$ In this case it is possible to prove that there
is a smooth surface $S\in |T-\pi ^{*}D|,$ with $D\equiv C_{0}+f,$ such that $%
S$ is isomorphic to the blow up of $\Bbb{P}^{2}$ at $9$ distinct points in
general position and $T_{|S}\simeq 4l-2l_{0}-l_{1}...-l_{8}$ (with the usual
notation). If we consider an existing smooth plane cubic curve passing
through $P_{0},...,P_{8}$ we have that this curve gives rise to a smooth
elliptic curve $C$ on $S$ such that $T_{|C}$ is not very ample ($\deg
(T_{|S}C)$ $=\deg (T_{|C})=2$), so that $T$ can not be very ample. In the
same way it is possible to show that the same approach is not successful for
$y=4,$ $h=3$ and also for $y=4$ and $h=5.$

On the other hand it is easy to prove that there exists a very ample rank $2$
vector bundles $\mathcal{E}_{-3}$ over $Y$ having $c_{1}(\mathcal{E}%
_{-3})\equiv 3C_{0}+5f$ and $c_{2}(\mathcal{E}_{-3})=5:$ indeed one can
simply take $\mathcal{E}_{-3}$ to be the direct sum of the two very ample
line bundles $L\equiv C_{0}+2f$ and $M\equiv 2C_{0}+3f$.

The following necessary condition for the very ampleness of rank $2$ vector
bundles $\mathcal{E},$ over any $\mathbf{F}_{e}$, having $c_{1}(\mathcal{E}%
)\equiv 3C_{0}+tf$ for some integer $t,$ can be established.

\begin{proposition}
\label{Brosius}Let $\mathcal{E}$ be a very ample rank $2$ vector bundle over
$\mathbf{F}_{e},$ such that $c_{1}(\mathcal{E})\equiv 3C_{0}+tf$ and $c_{2}(%
\mathcal{E})=k.$ Then: $h^{0}(\mathbf{F}_{e},\mathcal{E})\geq 7,$ $t\geq
3e+1 $, $k+e>t$ and there exists an exact sequence $0\rightarrow
L\rightarrow \mathcal{E}\rightarrow M\rightarrow 0$ where $L$ and $M$ are
line bundles such that $L\equiv 2C_{0}+(2t-2e-k)f$, and $M\equiv
C_{0}+(k-t+2e)f.$
\end{proposition}

\begin{proof}
As $\mathcal{E}$ is very ample, $(\Bbb{P}(\mathcal{E}),{}\mathcal{O}_{\Bbb{P}%
(\mathcal{E})}(1))$ is a scroll over $\mathbf{F}_{e}$ and it is known that
there are no such scrolls in $\Bbb{P}^{r}$ when $r\leq 5,$ see \cite{o}.
Moreover, if $\mathcal{E}$ is very ample then $c_{1}(\mathcal{E})$ is ample,
so that $t\geq 3e+1,$ see \cite[Corollary 2.18]{h}. Let $\rho :\mathbf{F}%
_{e}\to \Bbb{P}^{1}$ be the usual natural projection. The restriction ${%
\mathcal{E}}_{\mid _{f}}$ to any fibre $f$ of $\rho $ must also be very
ample. But $\mathcal{E}_{|f}=\mathcal{O}_{\Bbb{P}^{1}}(a)\oplus \mathcal{O}_{%
\Bbb{P}^{1}}(b)$ with $a+b=3$ as $c_{1}(\mathcal{E})f=3,$ therefore the only
possibility is $a=2,b=1$ for any fibre. By \cite[p. 155]{b}, Theorem 1,
there exists an exact sequence $0\rightarrow L\rightarrow \mathcal{E}
\rightarrow M\rightarrow 0$ such that $L+M\equiv c_{1}(\mathcal{E}),LM=c_{2}(%
\mathcal{E})=k$ and $L=\rho ^{*}[\rho _{*}(\mathcal{E}\otimes \mathcal{O}_{%
\mathbf{F}_{e}}(-2C_{0})]\otimes \mathcal{O}_{\mathbf{F}_{e}}(2C_{0})$. In
fact the zero-dimensional scheme $Z$ which is involved in the exact sequence
of \cite{b} in this case is empty. Indeed $\rho (\mathrm{Supp}(Z))$ would
coincide with the projection of the jumping lines for $\mathcal{E}$, but in
this case $\mathcal{E}$ is uniform on the ruling. It follows that $L\equiv
2C_{0}+\lambda f$ for some integer $\lambda $, and $M\equiv C_{0}+(t-\lambda
)f.$ As $LM=c_{2}(\mathcal{E})=k,$ it is $\lambda =2t-2e-k.$ The very
ampleness of $\mathcal{E}$ implies that $M$ is ample, hence $t-(2t-2e-k)>e,$
i.e. $k+e>t.$
\end{proof}

Proposition \ref{Brosius} shows that, if a vector bundle over $Y$ as $%
\mathcal{E}_{4}$ is very ample, then it is the extension of two line
bundles. However this fact does not help to prove the very ampleness of a
rank $2$ vector bundle by our techniques.

We can also establish a non-existence result that settles one more existence
question left open in \cite{bb1}.

\begin{corollary}
\label{finalno} There does not exist any linearly normal $3$-fold $X=\Bbb{P}(%
\mathcal{E}),$ embedded in $\Bbb{P}^{7}$ as linear scroll over $Y=$ $\mathbf{%
F}_{1}$, with degree $11$ and sectional genus $5$.
\end{corollary}

\begin{proof}
By contradiction, let us assume that $X$ exists in $\Bbb{P}^{7}$, hence $%
h^{0}(X,T)=8$ because $X$ is supposed to be linearly normal. As in the proof
of Lemma \ref{lemparete} $iii),$ let us consider a divisor $A\equiv C_{0}+xf$
on $Y$ with $x>>0$ and a smooth surface $S_{2}$ such that $S_{2}$ is
isomorphic to the blow up of $\Bbb{P}^{2}$ at $5x+12$ distinct points and $%
H^{0}(X,T)=H^{0}(S_{2},T_{|S_{2}}).$ Recall that we are assuming that $%
\mathcal{E}$ is very ample so that $\mathcal{E}\otimes A$ is very ample too;
note that here the position of the points on $\Bbb{P}^{2}$ it is not
important. As $T_{|S_{2}}\equiv (x+5)l-(x+1)l_{0}-l_{1}\dots -l_{5x+11}$ we
have that $h^{0}(S_{2},T_{|S_{2}})\geq 9:$ contradiction !
\end{proof}


\begin{thebibliography}{A-B-B}
\bibitem[A-B-B]{abb}  A.Alzati-M.Bertolini-G.M.Besana: ''Numerical criteria
for very ampleness of divisors on projective bundles over an elliptic
curve'' \textit{Can. J. Math.} \textbf{48}(6) (1996) p.1121-1137.

\bibitem[B]{b}  \medskip J.E.Brosius: ''Rank-2 Vector Bundles on a Ruled
Surface I'' \textit{Math. Ann}. \textbf{265} (1983) p.155-168.

\bibitem[B-B-1]{bb1}  G.Besana-A.Biancofiore: ''Degree 11 Manifolds of
Dimension Greater or equal to 3'' \textit{Forum Math.} \textbf{17 }(5)
(2005) p.711-733.

\bibitem[B-B-2]{bb2}  G.Besana-A.Biancofiore: ''Numerical Constraints for
Embedded Projective Manifolds'' \textit{Forum Math.} \textbf{17} (4) (2005)
p.613-636.

\bibitem[B-F]{bf}  G.M.Besana-M.L.Fania: ''The dimension of the Hilbert
scheme of special threefolds''. \textit{Comm. Algebra} \textbf{33} (10)
(2005), p.3811--3829.

\bibitem[B-S]{bs}  M.Beltrametti-A.J.Sommese: ''The adjunction theory of
complex projective varieties'' Expositions in Mathematics, vol 16, De
Gruyter, 1995.

\bibitem[B-D-S]{bds}  M.Beltrametti-S.Di Rocco-A.J.Sommese: ''On generation
of jets for vector bundles'' \textit{Rev. Mat. Complut.} \textbf{12} (1)
(1999) p.27-45.

\bibitem[Bu]{bu}  D.C.Butler: ''Normal generation of vector bundles over a
curve'' \textit{J. Differential Geom.} \textbf{39}(1) (1994) p.1-34.

\bibitem[C-M]{cm}  C. Ciliberto-R.Miranda: ''Degeneration of planar linear
systems'' \textit{J. reine angew. Math.} \textbf{501} (1998) p.191-220.

\bibitem[D-L]{dl}  G.F. del Busto-R.Lazarsfeld: ''Complex algebraic geometry
(Park City, UT, 1993)'', IAS/Park City Math. Ser., 3, p.161--219, Amer.
Math. Soc., Providence, RI, 1997.

\bibitem[F-L-1]{fl1}  L.Fania-E.L.Livorni: ''Degree Nine Manifolds of
Dimension Greater than or Equal to 3'' \textit{Math. Nachr}. \textbf{169}
(1994) p.117-134.

\bibitem[F-L-2]{fl2}  L.Fania-E.L.Livorni: ''Degree ten manifolds of
dimension $n$ greater than or equal to 3'' \textit{Math. Nachr.} \textbf{188}
(1997) p.79--108.

\bibitem[G-H]{gh}  P.Griffiths-J.Harris: ''Principles of Algebraic
Geometry'' Wiley Interscience New York 1994.

\bibitem[H]{h}  R.Hartshorne: ''Algebraic Geometry'' Springer Verlag 1977.

\bibitem[H-R]{hr}  A.Holme-J.Roberts: ''On the embeddings of projective
varieties''. Algebraic geometry (Sundance, UT, 1986), \textit{Lecture Notes
in Math}. \textbf{1311}, Springer, Berlin, (1988) p.118--146.

\bibitem[M]{m}  Y.Miyaoka: ''The Chern class and Kodaira dimension of a
minimal variety'' In: Algebraic Geometry, Sendai 1985, \textit{Adv. Stud.
Pure Math.} \textbf{10}, Amer. Math. Soc., (1987), p.449-476.

\bibitem[O]{o}  G.Ottaviani: ''On 3-Folds in $\Bbb{P}^{5}$ which are
Scrolls'' \textit{Ann. Sc. Norm.\ Sup. Pisa}, Serie IV, \textbf{19} (3)
(1992) p.451-471.
\end{thebibliography}
\end{document}